\documentclass[11pt]{article}
\usepackage{amssymb,amsmath,amsfonts,mathrsfs,amsthm,epsfig,latexsym,color}
\usepackage{enumitem,geometry}
\usepackage{overpic}

\makeatletter
\newcommand{\subsectionruninhead}{\@startsection{subsection}{2}{0mm}
{-\baselineskip}{-0mm}{\bf\large}}
\newcommand{\subsubsectionruninhead}{\@startsection{subsubsection}{3}{0mm}
{-\baselineskip}{-0mm}{\bf\normalsize}}
\makeatother

\newtheorem*{theorem*}{Theorem}
\newtheorem*{proof*}{Proof}

\newtheorem*{proposition*}{Proposition}
\newtheorem*{corollary*}{Corollary}
\newtheorem*{claim*}{Claim}
\newtheorem{theorem}{Theorem}
\newtheorem{proposition}{Proposition}[section]

\newtheorem{lemma}[proposition]{Lemma}

\newtheorem{claim}[proposition]{Claim}
\theoremstyle{definition}
\newtheorem{definition}[proposition]{Definition}
\theoremstyle{remark}
\newtheorem{remark}[proposition]{Remark}

\numberwithin{equation}{section}

   \def\TT{{\mathbb T}}

\newcommand{\leb}{\operatorname{Leb}}

\newcommand{\supp}{\operatorname{Supp}}
\newcommand{\diff}{\operatorname{Diff}}
\newcommand{\id}{\operatorname{Id}}
\newcommand{\dd}{\operatorname{d}}

\newcommand{\orb}{\operatorname{Orb}}
\newcommand{\width}{\operatorname{Wd}}

\begin{document}
\title{A robustly transitive diffeomorphism of Kan's type}
\author{CHENG CHENG, SHAOBO GAN AND YI SHI}
\date{\today}
\maketitle
\begin{abstract}
We construct a family of partially hyperbolic skew-product diffeomorphisms on $\mathbb{T}^3$ that are robustly transitive and admitting two physical measures with intermingled basins. In particularly, all these diffeomorphisms are not topologically mixing. Moreover, for every such example, it exhibits a dichotomy under perturbation: every perturbation of such example either has a unique physical measure and is robustly topologically mixing, or has two physical measures with intermingled basins.
\end{abstract}

\section{Introduction}\label{sec:introduction}

Let $M$ be a compact Riemannian manifold of dimension $d$, and $f:M\rightarrow M$ be a diffeomorphism. We call an $f$-invariant Borel probability measure $\mu$ a \emph{physical measure} if the set $B(\mu)$ defined as
\begin{displaymath}
B(\mu)=\{x\in M:\lim_{n\rightarrow\infty}\frac{1}{n}\sum_{i=0}^{n-1}\varphi(f^i(x))=\int\varphi d\mu,\textrm{for all}\ \varphi\in C^0(M)\}
\end{displaymath}
has positive Lebesgue measure. The set $B(\mu)$ is called the \emph{basin of $\mu$}. These notions were introduced by Sinai in \cite{S}, where it was proved the existence of physical measures for Anosov diffeomorphisms. Subsequently, in \cite{R} and \cite{BR}, Bowen and Ruelle showed that for hyperbolic diffeomorphisms and flows, there exist a finite number of physical measures whose basins are all essentially open (i.e., open module a set of null Lebesgue measure), and the union of these basins cover a full Lebesgue measure subset of the ambient manifold. For non-hyperbolic diffeomorphisms, in \cite{P} Palis conjectured that every system can be approximated by systems with a finite number of physical measures whose basins unite to cover a full Lebesgue measure subset of the whole manifold. In this term, as a positive step, in \cite{BV} and \cite{ABV}, Alves, Bonatti and Viana proved the existence and finiteness of physical measures for a large class of non-hyperbolic diffeomorphisms, namely partially hyperbolic systems with mostly contracting or mostly expanding center.

For hyperbolic systems, basins of physical measures are essentially open. However, the boundaries of the basins are often fractal sets (e.g., see \cite{DGOY}). In \cite{K}, Kan constructed for the first time examples of partially hyperbolic endomorphisms defined on the 2-dimensional cylinder with exactly two physical measures whose basins are \emph{intermingled}. Here we say two physical measures $\mu$ and $\nu$ have \emph{intermingled basins} or their basins are \emph{intermingled} if for any open set $U$ in the manifold, we have $\leb(B(\mu)\cap U)>0$ and $\leb(B(\nu)\cap U)>0$.

Kan's examples robustly admit two physical measures with intermingled basins within boundary preserving systems. His construction was extended by Ilyashenko, Kleptsyn and Saltykov\cite{IKS}. Kan also constructed partially hyperbolic examples on the thickened torus with two physical measures whose basins are intermingled. Then in \cite{KS}, Kleptsyn and Saltykov proved that the boundaries of basins of physical measures in Kan's examples have Hausdorff dimension strictly smaller than the dimension of the phase space.

One would ask whether there exist more than two physical measures with their basins intermingled. In \cite{DVY}, Dolgopyat, Viana and Yang constructed examples on $\mathbb{T}^2\times\mathbb{S}^2$ admitting arbitrarily finitely many physical measures with intermingled basins. And recently, Bonatti and Potrie (\cite{BP}) give examples on $\mathbb{T}^3$ with arbitrarily finitely many physical measures whose basins are intermingled. Their examples are partially hyperbolic but not strongly partially hyperbolic, that is, the center stable bundle $E^{cs}$ can not be split to two invariant subbundles, which is different from the previous examples.

It is well known that Kan's examples can be extended to boundaryless manifolds such as $\mathbb{T}^3$, following the same arguments. In this case, however, Ures and Vasquez proved recently in \cite{UV} that Kan's examples can not be robust. In fact, they proved for $C^r(r\geq2)$-partially hyperbolic, dynamically coherent diffeomorphisms of $\mathbb{T}^3$ with compact center leaves, if two physical measures are intermingled, then they must both support in periodic tori which are $s,u$-saturated. Thus these systems can not be accessible, and the number of physical measures with intermingled basins can not be larger than two. Recall that in \cite{BHHTU}, Burns, Herts, Herts, Talitskaya and Ures proved that accessibility property is $C^1$ open and $C^{\infty}$ dense among partially hyperbolic diffeomorphisms with one-dimensional center. This implies that partially hyperbolic diffeomorphisms on $\mathbb{T}^3$ with intermingled basins can not form a $C^r$ open set.

It's also a natural question to ask whether Kan's examples can be transitive or robustly transitive. In Chapter 11 of \cite{BDV}, it is showed that with two more assumptions, Kan's examples on 2-dimensional cylinder can be transitive. And recently, in \cite{O}, Okunev proved that for the set of $C^r$ partially hyperbolic diffeomorphisms which are skew products over transitive Anosov diffeomorphisms and fibered by $\mathbb{S}^1$, there exists a residual subset $\mathcal{R}$, such that for every $f\in\mathcal{R}$ either it is transitive or its non-wandering set has zero Lebesgue measure.

In this paper, we append some extra constructions to Kan's examples on $\mathbb{T}^3$, which lead them to be robustly transitive. First, we need to give the precise definitions of Kan's examples on $\mathbb{T}^2\times[0,1]$ and $\mathbb{T}^3=\mathbb{T}^2\times \mathbb{S}^1$, where in this paper $\mathbb{S}^1$ will be identified with $\mathbb{R}/2\mathbb{Z}$. These definitions were abstracted from Kan's original examples (see \cite{BDV}), which will guarantee the existence of two physical measures with intermingled basins. In this paper, we will use ${\diff^{1+}(M)}$ to denote the set of diffeomorphisms on $M$ with H$\rm\ddot{o}$lder derivative.

\begin{definition}\label{def:Kan's example with boundary}
Let $M=\mathbb{T}^2\times[0,1]$. $K\in\diff^{1+}(M)$ is called a \emph{Kan's example} if it is a skew product diffeomorphism $K(x,\theta)=(\mathbb{A}x,\phi(x,\theta))$, where $\mathbb{A}:\mathbb{T}^2\rightarrow\mathbb{T}^2$ is a hyperbolic automorphism, and $\phi:M\rightarrow[0,1]$ is $C^{1+}$, satisfying the following conditions:
\begin{itemize}
\item[(K1)] For every $x\in\mathbb{T}^2,\phi(x,0)=0$ and $\phi(x,1)=1.$
\item[(K2)] For $r,s\in\mathbb{T}^2$, fixed points of $\mathbb{A}$, the map $\phi(r,\cdot)$ (resp. $\phi(s,\cdot)$) is Morse--Smale and has exactly two fixed points, a source (resp. sink) at $\theta=1$ and a sink (resp. source) at $\theta=0$.
\item[(K3)] For every $(x,\theta)\in M$, $\parallel \mathbb{A}^{-1}\parallel^{-1}<|\partial_{\theta}\phi(x,\theta)|<\parallel \mathbb{A}\parallel.$
\item[(K4)] $\int_{\mathbb{T}^2}\log|\partial_{\theta}\phi(x,0)|dx<0$ and $\int_{\mathbb{T}^2}\log|\partial_{\theta}\phi(x,1)|dx<0$.
\end{itemize}
The set of $C^{1+}$ Kan's examples on $\mathbb{T}^2\times[0,1]$ will be denoted by $\mathscr{K}^{1+}(\mathbb{T}^2\times[0,1]).$
\end{definition}

For the boundaryless case, we have naturally the following definition (see also \cite{UV}).

\begin{definition}\label{def:Kan's example on boundaryless manifolds}
Let $M=\mathbb{T}^3=\mathbb{T}^2\times \mathbb{S}^1$. $K\in\diff^{1+}(M) $ is called a \emph{Kan's example} if it is a skew product diffeomorphism $K(x,\theta)=(\mathbb{A}x,\phi(x,\theta))$, where $\mathbb{A}:\mathbb{T}^2\rightarrow\mathbb{T}^2$ is a hyperbolic automorphism, and $\phi:M\rightarrow\mathbb{S}^1$ is $C^{1+}$, satisfying the following conditions:
\begin{itemize}
\item[(K1)] For every $x\in\mathbb{T}^2,\phi(x,0)=0$ and $\phi(x,1)=1.$
\item[(K2)] For $r,s\in\mathbb{T}^2$, fixed points of $\mathbb{A}$, the map $\phi(r,\cdot)$ (resp. $\phi(s,\cdot)$) is Morse--Smale and has exactly two fixed points, a source (resp. sink) at $\theta=1$ and a sink(resp. source) at $\theta=0$.
\item[(K3)] For every $(x,\theta)\in M$, $\parallel \mathbb{A}^{-1}\parallel^{-1}<|\partial_{\theta}\phi(x,\theta)|<\parallel \mathbb{A}\parallel.$
\item[(K4)] $\int_{\mathbb{T}^2}\log|\partial_{\theta}\phi(x,0)|dx<0$ and $\int_{\mathbb{T}^2}\log|\partial_{\theta}\phi(x,1)|dx<0$.
\end{itemize}
The set of $C^{1+}$ Kan's examples on $\mathbb{T}^3$ will be denoted by $\mathscr{K}^{1+}(\mathbb{T}^3).$
\end{definition}

Noting that Kan's aim is to show the existence of intermingle property, we may generalize Kan's examples as in \cite{UV} to the following Kan--like diffeomorphisms.

\begin{definition}[Kan--like diffeomorphisms]\label{def:Kan-lian diffeo}
Let $M$ be some compact Riemannian manifold. A $C^{1+}$-partially hyperbolic diffeomorphism $f:M\rightarrow M$ is called a \emph{Kan--like diffeomorphism} if it admits at least two hyperbolic physical measures such that their basins are intermingled. The set of Kan--like diffeomorphisms is denoted by $\mathscr{K}_l^{1+}(M)$.
\end{definition}

We have mentioned that a beautiful description of Kan--like diffeomorphisms on $\TT^3$ has been given in \cite{UV}. They proved that every partially hyperbolic Kan--like diffeomorphism on $\TT^3$ must admit two $s,u$-saturated periodic tori as supports of physical measures.

Our example is $C^{\infty}$ smooth, and can be got by an arbitrarily small perturbation of $\mathbb{A}\times (-\id_{\mathbb{S}^1})$, where $\mathbb{A}$ is a hyperbolic automorphism on $\mathbb{T}^2$, and $-\id_{\mathbb{S}^1}$ is an orientation reversing isometry on $\mathbb{S}^1=\mathbb{R}/2\mathbb{Z}$, which admits $0$ and $1(\equiv-1)$ as two fixed points.

\begin{theorem}\label{theo:main theorem}
Let $\mathbb{A}:\mathbb{T}^2\rightarrow\mathbb{T}^2$ be a hyperbolic automorphism. Then there exists a $C^{\infty}$-arc $\{K_t\}_{t\in[0,1]}$ of $C^{\infty}$-diffeomorphisms on $\mathbb{T}^3=\mathbb{T}^2\times \mathbb{S}^1$, such that
\begin{enumerate}
\item $K_0=\mathbb{A}\times (-\id_{\mathbb{S}^1})$.
\item For every $t\in(0,1]$, $K_t\in\mathscr{K}^{1+}(\mathbb{T}^3)$ has two physical measures with intermingled basins and is $C^1$-robustly transitive, but not topologically mixing.
\end{enumerate}

Moreover, for every $t\in(0,1]$, there exists a $C^1$-neighborhood $\mathcal{U}_t$ of $K_t$, such that for every  $g\in\mathcal{U}_t\cap\diff^{1+}(\mathbb{T}^3)$,
\begin{itemize}
\item either $g$ has two physical measures with intermingled basins and $g\in\mathscr{K}_l^{1+}(\mathbb{T}^3)$,
\item or $g$ has a unique physical measure and $g$ is robustly topologically mixing.
\end{itemize}
\end{theorem}

\begin{remark}
A diffeomorphism is $C^r$-robustly transitive if it has a $C^r$-neighborhood consisting of transitive diffeomorphisms. All the known examples of $C^1$-robustly transitive diffeomorphisms are topologically mixing. So our example is the first $C^1$-robustly transitive diffeomorphism that are not topologically mixing, which gives a negative answer to a question in \cite{AC}.
\end{remark}
\begin{remark}
We would like to point out here that if $f\in\mathscr{K}^{1}(\mathbb{T}^3)$ preserves the orientation of the center fibers, then $f$ can not be robustly chain transitive. To see this, notice that there are two open sets $U=\mathbb{T}^2\times(-1,0)$ and $V=\mathbb{T}^2\times(0,1)$ that are both $f$-invariant. Define $\mathcal{X}$ to be a Morse--Smale vector field on $\mathbb{S}^1$ with a unique sink at $\theta=-0.5$ and a unique source at $\theta=0.5$. Then $f_t=(\id_{\mathbb{T}^2}\times\mathcal{X}_t)\circ f$ is a family of diffeomorphisms isotopic to $f$, such that $\overline{f_t(U)}\subset U$ for every $t>0$. This means that $U$ is a trapping region of $f_t$ for every $t>0$, and it follows that $f_t$ is not chain transitive.
\end{remark}

To prove Theorem~\ref{theo:main theorem}, we will firstly give the constructions in Section~\ref{sec:construction}, and then prove the corresponding properties via a sequence of propositions in Section~\ref{sec:proofs}. In Proposition~\ref{prop:robust transitivity}, we prove the examples we constructed are robustly transitive. In Proposition~\ref{prop:two physical measures}, we prove that if the perturbed system has two physical measures, then their basins are intermingled. And in Proposition~\ref{prop:unique physical measure}, we prove that if the perturbed system has a unique physical measure, then the system is robustly topologically mixing. We should point out here that instead of proving Proposition~\ref{prop:two physical measures}, we will actually prove the following theorem which is more general than Proposition~\ref{prop:two physical measures}.

\begin{theorem}\label{theo:two-pm}
For any $f\in\mathscr{K}^{1+}(\mathbb{T}^3)$, there exists a $C^1$-neighborhood ${\cal U}$ of $f$, such that for any $g\in{\cal U}\cap\diff^{1+}(\mathbb{T}^3)$, if $g$ has more than one physical measures, then $g\in\mathscr{K}_l^{1+}(\mathbb{T}^3)$.
\end{theorem}

\begin{remark}
Generally, this theorem doesn't work for $f\in\mathscr{K}_l^{1+}(\mathbb{T}^3)$ since our proof of Theorem~\ref{theo:two-pm} relies on property (K2) in the definition of Kan's examples.
\end{remark}

Another byproduct of our main result is the following corollary.

\begin{corollary*}
There exists a partially hyperbolic skew product diffeomorphism $g\in\diff^{\infty}(\mathbb{T}^3)$ such that $g$ is robustly topologically mixing, but neither the stable nor the unstable foliation of $g$ is minimal.
\end{corollary*}

\begin{remark}\label{remark}
Since $g$ in the above corollary is a skew product diffeomorphism over a linear Anosov automorphism of the torus, it must have normally hyperbolic periodic center compact leaves. Moreover, $g$ can also be constructed such that it preserves the orientation of all three invariant bundles of $g$'s partially hyperbolic splitting. However, for an open dense subset of such kind systems, Theorem 1.6 of \cite{BDU} showed that both the stable and the unstable foliations are robustly minimal.
\end{remark}

Notice that Kan constructed in \cite{K} an endomorphism on the cylinder admitting two physical measures with intermingled basins. One can define Kan's map on the cylinder like Definition \ref{def:Kan's example with boundary} (see \cite{BDV}). Recently, Gan and Shi \cite{GS} proved that all these Kan's maps are robustly transitive in the setting of boundary preserving case.

~

\noindent\emph{Outline of the paper}: Our main results are all stated in Section~\ref{sec:introduction}. In Section~\ref{sec:pre}, we recall some materials that will be necessary for the proofs. In Section~\ref{sec:construction}, we show how to construct the examples we need. The core step of this section is the construction of blender-horseshoes, which is kind of technical and complicated. The readers could skip this part for the first reading, because we will summarize the main properties of our examples at the beginning of Section~\ref{sec:proofs}. Finally in the remaining of Section~\ref{sec:proofs}, we will give the proofs of our results.

\section{Preliminaries}\label{sec:pre}

In this section, we will give some definitions and notations that will be used in this paper. Firstly, we will introduce the definition of Gibbs $u$-states, which was introduced in \cite{PS}, and has been proved to be a powerful tool for studying physical measures. Then, we will define a special class of diffeomorphisms, now called systems with mostly contracting center. These systems are firstly introduced and studied by Bonatti and Viana\cite{BV}. In a word, they have shown the existence and finiteness of physical measures for such systems.

Secondly, we will give a brief introduction to blender-horseshoes and blenders. The conception of blenders is firstly introduced by Bonatti and D\'iaz \cite{BD1}, and it has been proved to be a remarkable idea and a powerful tool for proving persistence of cycles and transitivity (see \cite{BD1,BD2}). Then in \cite{BD3}, they raise another conception, called blender-horseshoes, and prove it to be a special class of blenders.

\subsection{Gibbs \boldmath{$u$-}states and systems with mostly contracting center}
Given a compact Riemannian manifold $M$, a diffeomorphism $f:M\rightarrow M$ is called \emph{partially hyperbolic} if the tangent bundle admits a dominated splitting $TM=E^s\oplus E^c\oplus E^u,$ such that $Df|_{E^s}$ is uniformly contracting, $Df|_{E^u}$ is uniformly expanding and $Df|_{E^c}$ lies between them:
\begin{displaymath}
\parallel Df|_{E^s(x)}\parallel<\parallel Df^{-1}|_{E^c(f(x))}\parallel^{-1},\quad \parallel Df|_{E^c(x)}\parallel<\parallel Df^{-1}|_{E^u(f(x))}\parallel^{-1},\quad\textrm{for all}\ x\in M.
\end{displaymath}

The \emph{stable bundle} $E^s$ and \emph{unstable bundle} $E^u$ are automatically integrable: there exists a unique foliation $\mathcal{F}^u$ (resp. $\mathcal{F}^s$) tangent to $E^u$ (resp. $E^s$) at every point. $\mathcal{F}^u$ and $\mathcal{F}^s$ are called \emph{unstable foliation} and \emph{stable foliation}, respectively. A subset $\Lambda$ of $M$ is called \emph{$u$-saturated} (resp. \emph{$s$-saturated}) if it consists of entire strong unstable (resp. strong stable) leaves. Any compact disk embedded in a leaf of $\mathcal{F}^u$ is called a \emph{$u$-disk}.

An $f$-invariant probability measure $\mu$ is called a \emph{Gibbs $u$-state} if the conditional probabilities of $\mu$ along $\mathcal{F}^u$ are absolutely continuous with respect to the Lebesgue measures on the leaves.

The following proposition summarizes the basic properties of Gibbs $u$-states.

\begin{proposition}[\cite{BDV,PS}]\label{prop:gibbs ustate basic properties}
\begin{itemize}
\item[(1)] For every $u$-disk $D$, every accumulation point of the sequence of probability measures
$$\mu_n=\frac{1}{n}\sum_{i=0}^{n-1}f^i_*\left(\frac{m_D}{m_D(D)}\right)$$
is a Gibbs $u$-state, where $m_D$ is the Lebesgue measure restricted to the disk $D$.
\item[(2)] The support of every Gibbs $u$-state is $u$-saturated.
\item[(3)] Every ergodic component of a Gibbs $u$-state is still a Gibbs $u$-state.
\item[(4)] For Lebesgue almost every point $x$ in any $u$-disk, every accumulation point of $\frac{1}{n}\sum_{i=0}^{n-1}\delta_{f^i(x)}$ is a Gibbs $u$-state.
\item[(5)] Every physical measure is a Gibbs $u$-state. Conversely, every ergodic Gibbs $u$-state whose center Lyapunov exponents are all negative is a physical measure.
\end{itemize}
\end{proposition}

We say a $C^{1+}$-partially hyperbolic diffeomorphism $f:M\rightarrow M$ has \emph{mostly contracting center} if for any $u$-disk $D^u$, we have
\begin{displaymath}
\limsup_{n\rightarrow\infty}\frac{1}{n}\log\parallel Df^n(x)|_{E^c(x)}\parallel<0
\end{displaymath}
for every $x$ in a subset $D_0^u\subset D^u$ with positive Lebesgue measure.

\begin{proposition}[\cite{A,VY}]\label{prop:mcc equivalent property}
A diffeomorphism $f$ has mostly contracting center if and only if the center Lyapunov exponents of all ergodic
Gibbs $u$-states of $f$ are negative.
Moreover, if $f$ has mostly contracting center, then
\begin{itemize}
\item there exist finitely many ergodic Gibbs $u$-states for $f$, and each of them is an ergodic physical measure.
\item the supports of the ergodic Gibbs $u$-states of $f$ are pairwise disjoint.
\end{itemize}
\end{proposition}

Proposition~\ref{prop:mcc equivalent property} gives the equivalent characterization of systems with mostly contracting center. Next lemma shows that mostly contracting systems are $C^1$ open.

\begin{lemma}[\cite{Y}, Theorem~B]\label{lem:openness of mcc}
Let $M$ be some compact Riemannian manifold. Assume $f$ is any $C^{1+}$-partially hyperbolic diffeomorphism on $M$ having mostly contracting center, then there exists a $C^1$-neighbourhood $\mathcal{U}$ of $f$, such that every $g\in\mathcal{U}\cap\diff^{1+}(M)$ also has mostly contracting center.
\end{lemma}

\subsection{\boldmath{Blender-horseshoes and $cu$-}blenders}
Suppose $f:M\rightarrow M$ is a $C^1$-diffeomorphism, where $M$ is an $n$-dimensional compact Riemannian manifold, $n\geq3$. In this subsection, we write $n=n_s+1+n_u,$ where $n_s,n_u\geq1.$ Consider a $C^1$-embedding $\varphi: D^s\times[-1,1]\times D^u\rightarrow M$, where $D^s=[-1,1]^{n_s}$ and $D^u=[-1,1]^{n_u}$. Denote by $\Gamma$ the image $\varphi(D^s\times[-1,1]\times D^u)$. In $\Gamma$, we take the coordinate $(x_s,x_c,x_u), x_s\in D^s, x_c\in[-1,1], x^u\in D^u.$ In the manifold $M$, we take a metric $\parallel\cdot\parallel$ that induces in the cube $\Gamma$ the product of the usual euclidean metrics in $D^s,[-1,1]$ and $D^u$. Denote four specific parts of $\partial \Gamma$ as follows:
\begin{displaymath}
\begin{array}{rcl}
\partial^u\Gamma&=&D^s\times\partial([-1,1]\times D^u).\\
\partial^{uu}\Gamma&=&D^s\times[-1,1]\times\partial D^u.\\
\partial^{ss}\Gamma&=&(\partial D^s)\times[-1,1]\times D^u.\\
\partial^c\Gamma&=&D^s\times\{-1,1\}\times D^u.
\end{array}
\end{displaymath}

Given a $k$-dimensional plane $\Pi\subset\mathbb{R}^n$ and $\varepsilon>0$, define the \emph{$\varepsilon$-cone around $\Pi$} by $C_{\varepsilon}(\Pi)=\{u\in \mathbb{R}^n: u=v+w,w\in\Pi,w\in\Pi^{\bot},\parallel w\parallel\leq\varepsilon\parallel v\parallel\}$.

For $\varepsilon>0,$ define $\mathcal{C}^{uu}_{\varepsilon},\mathcal{C}^u_{\varepsilon}$ and $\mathcal{C}^{ss}_{\varepsilon}$ to be  conefields of size $\varepsilon$ around the tangent space of the families of disks $\{x_s,x_c\}\times D^u$, $\{x_s\}\times[-1,1]\times D^u$ and $D^s\times\{x_c,x_u\}$, respectively.

Now fix an $\varepsilon>0$. We say that an $n_u$-disk(resp. $n_s$-disk) $\Delta$ in $\Gamma$ is a \emph{vertical disk}(resp. horizontal disk) if $\Delta$ is tangent to $\mathcal{C}^{uu}_{\varepsilon}$(resp. $\mathcal{C}^{ss}_{\varepsilon}$) and its boundary is contained in $\partial^{uu}\Gamma$(resp. $\partial^{ss}\Gamma$).

Let $\Delta\subset \Gamma$ be a horizontal disk such that $\Delta\cap\partial^u\Gamma=\varnothing$. Then there are two different homotopy classes of vertical disks through $\Gamma$ disjoint from $\Delta$. A vertical disk that does not intersect $\Delta$ is \emph{at the right of $\Delta$} if it is in the homotopy class of $\{0_s\}\times\{1\}\times D^u$. Otherwise we say that it is \emph{at the left of $\Delta$}.

Let $\Delta$ be some horizontal disk. A \emph{vertical strip} (\emph{at the right of $\Delta$}) is defined to be an embedding $\Phi:[0,1]\times D^u\rightarrow\Gamma$ such that for every $t\in[0,1]$, $\Phi(\{t\}\times D^u)$ is a vertical disk (at the right of $\Delta$) and $\Phi([0,1]\times D^u)$ is tangent to $\mathcal{C}^{u}_{\varepsilon}$. We denote $\Phi([0,1]\times D^u)$ by $\mathcal{S}$. The width of a vertical strip $\mathcal{S}$, denoted by $\width(\mathcal{S})$, is defined to be the minimum of the length of arcs contained in $\mathcal{S}$ and connecting $\Phi(\{0\}\times D^u)$ and $\Phi(\{1\}\times D^u)$.

\begin{definition}[Blender-horseshoes]\label{def:blenderhorseshoe}
We call $\Lambda=\cap_{i\in\mathbb{Z}}f^i(\Gamma)$, the maximal invariant set in $\Gamma$, a \emph{blender-horseshoe} if the following conditions hold(with respect to some local coordinate system)
\begin{itemize}
\item[(1)] $f(\Gamma)\cap(\mathbb{R}^s\times\mathbb{R}\times[-1,1]^u)$ consists of two components, denoted by $f(\tilde{\Gamma}_1)$ and $f(\tilde{\Gamma}_2)$. Moreover, $f(\tilde{\Gamma}_1)\cup f(\tilde{\Gamma}_2)\subset(-1,1)^s\times\mathbb{R}\times[-1,1]^u$, and $(\tilde{\Gamma}_1\cup\tilde{\Gamma}_2)\cap\partial^{uu}\Gamma=\varnothing.$
\item[(2)] There exists $\varepsilon_0>0$ such that the conefields $\mathcal{C}^{uu}_{\varepsilon_0},\mathcal{C}^u_{\varepsilon_0}$ and $\mathcal{C}^{ss}_{\varepsilon_0}$ are $Df$-invariant. That is, for $x\in f(\tilde{\Gamma}_1)\cup f(\tilde{\Gamma}_2)$, $Df^{-1}(\mathcal{C}^{ss}_{\varepsilon_0}(x))\subsetneqq \mathcal{C}^{ss}_{\varepsilon_0}(f^{-1}x)$, and for $x\in\tilde{\Gamma}_1\cup\tilde{\Gamma}_2$, $Df(\mathcal{C}^{uu}_{\varepsilon_0}(x))\subsetneqq \mathcal{C}^{uu}_{\varepsilon_0}(f(x))$ and $Df(\mathcal{C}^{u}_{\varepsilon_0}(x))\subsetneqq \mathcal{C}^{u}_{\varepsilon_0}(f(x))$. Moreover, $\mathcal{C}^{ss}_{\varepsilon_0}$ and $\mathcal{C}^{u}_{\varepsilon_0}$ are uniformly contracting and expanding under $Df$, respectively.
\item[(3)] Denote by $\Gamma_i=f^{-1}(f(\tilde{\Gamma}_i)\cap\Gamma), i=1,2.$ Then $\Gamma_1\cup\Gamma_2$ is disjoint from $\partial^u\Gamma$ and $f(\Gamma_1)\cup f(\Gamma_2)$ is disjoint from $\partial^{ss}\Gamma$. It can be deduced that $\{\Gamma_1,\Gamma_2\}$ forms a Markov partition of $\Lambda$. Denote by $P_1$ and $P_2$ to be the unique hyperbolic fixed points contained in $\Gamma_1$ and $\Gamma_2$, respectively.
\item[(4)] For any $x\in\Lambda$, define $W^{\sigma}_{loc}(x)$ to be the connected component of $W^{\sigma}(x)\cap \Gamma$ containing $x$, where $\sigma=s,u,uu$. Then for any vertical disks $\Delta$ and $\Delta'$ such that $\Delta\cap W^s_{loc}(P_1)\neq\varnothing$ and $\Delta'\cap W^s_{loc}(P_2)\neq\varnothing$, we have $\Delta\cap\partial^c\Gamma=\Delta'\cap\partial^c\Gamma=\Delta\cap\Delta'=\varnothing$.
\item[(5)] If $\Delta$ is a vertical disk at the right (resp. left) of $W^s_{loc}(P_i)$, then $f(\Delta\cap\Gamma_i)$ is a vertical disk at the right (resp. left) of $W^s_{loc}(P_i)$, where $i=1,2$. If $\Delta$ is at left of $W^s_{loc}(P_1)$ or $\Delta\cap W^s_{loc}(P_1)\neq\varnothing$, then $f(\Delta\cap\Gamma_2)$ is a vertical disk at the left of $W^s_{loc}(P_1)$. If $\Delta$ is at right of $W^s_{loc}(P_2)$ or $\Delta\cap W^s_{loc}(P_2)\neq\varnothing$, then $f(\Delta\cap\Gamma_1)$ is a vertical disk at the right of $W^s_{loc}(P_2)$.
\item[(6)] If $\Delta$ is a vertical disk in between $W^s_{loc}(P_1)$ and $W^s_{loc}(P_2)$, then either $f(\Delta\cap\Gamma_1)$ or $f(\Delta\cap\Gamma_2)$ is a vertical disk in between $W^s_{loc}(P_1)$ and $W^s_{loc}(P_2)$.
\end{itemize}
\end{definition}

\begin{remark}
The cube $\Gamma$ is called the \emph{reference cube} of the blender-horseshoe $\Lambda$. And $P_1,P_2$ are called the \emph{reference saddles}. For more details of Definition~\ref{def:blenderhorseshoe}, see \cite{BD3}.
\end{remark}

\begin{definition}[$cu$-blenders]\label{def:blender}
Let $f:M\rightarrow M$ be a $C^1$-diffeomorphism and $\Lambda$ be a transitive hyperbolic set with stable index $n_s$. Then $(\Lambda,f)$ is called a \emph{$cu$-blender} if there are a $C^1$-neighbourhood $\mathscr{U}$ of $f$ and a $C^1$-open set $\mathcal{D}$ of embeddings of $n_u$-dimensional disks $\Delta$ into $M$ such that for every $g\in\mathscr{U}$ and every $\Delta\in\mathcal{D}$, $\Delta\cap W^s_{loc}(\Lambda_g)\neq\varnothing$. Then $\mathcal{D}$ is called the \emph{superposition region} of the blender, and $\mathbb{D}=\cup_{\Delta\in\mathcal{D}}\Delta$ is called the \emph{superposition domain} of the blender. We also say $\Lambda$ is a $cu$-blender if there is no ambiguity.
\end{definition}

\begin{lemma}\label{lem:blenderhorseshoe and cu blender}
$(1)$(\cite{BD3},Lemma~3.9) Let $f:M\rightarrow M$ be a $C^1$-diffeomorphism and $\Lambda$ be a blender-horseshoe with reference cube $\Gamma$ and reference saddles $P_1$ and $P_2$. Then there exists a $C^1$- neighbourhood $\mathcal{U}$ of $f$ such that for any $g\in\mathcal{U}$, $\Lambda_g$ is a blender-horseshoe with reference cube $\Gamma$ and reference saddles $P_{1,g}$ and $P_{2,g}$.

$(2)$(\cite{BD3},Remark~3.10) The blender-horseshoe $\Lambda$ is also a $cu$-blender, and the superposition region contains the set of the vertical disks lying in between $W^s_{loc}(P_1)$ and $W^s_{loc}(P_2)$.
\end{lemma}

The following lemma can be deduced from the definition of $cu$-blenders and Palis' $\lambda$-lemma.

\begin{lemma}\label{lem:blender}
Suppose $(\Lambda,f)$ is a $cu$-blender, $P\in\Lambda$ is a hyperbolic fixed point, and $Q$ is a hyperbolic fixed point of index $n_s+1$ such that $W^u(Q)$ contains a vertical disk that belongs to the superposition region $\mathcal{D}$, then there exists a $C^1$-neighbourhood $\mathcal{U}$ of $f$, such that for any $g\in\mathcal{U}$, $W^s(Q_g,g)\subset\overline{W^s(P_g,g)}$.
\end{lemma}
\begin{proof}
Fix $\mathcal{U}$ small enough such that for any $g\in\mathcal{U}$, $\Lambda_g,P_g,Q_g$ are all well defined and $(\Lambda_g,g)$ is a $cu$-blender. Then since $\Lambda_g$ is transitive, $\Lambda_g\subset \overline{W^s(P_g,g)}$. Hence, we have that $W^s_{loc}(\Lambda_g,g)\subset \overline{W^s(P_g,g)}$. Since $W^u(Q)$ contains a vertical disk, denoted by $\Delta$, that belongs to the superposition region $\mathcal{D}$, then $W^u(Q_g,g)$ also contains a vertical disk, denoted by $\Delta_g$, that belongs to $\mathcal{D}$, by the continuity of (the compact part of) $W^u(Q_g,g)$ with respect to $g$  and the openness of $\mathcal{D}$.

Now for any $x\in W^s(Q_g,g)$, and any open neighbourhood $V$ of $x$, there exists $N>0$ large enough such that $g^N(V)$ contains a vertical disk $\Delta'_g$ close enough to $\Delta_g$ such that $\Delta'_g\in\mathcal{D}$. Therefore, $g^N(V)\cap W^s_{loc}(\Lambda_g,g)\neq\varnothing$, which implies $g^N(V)\cap \overline{W^s(P_g,g)}\neq\varnothing$. Now the conclusion follows by using the $g$-invariance of $W^s(P_g,g)$.
\end{proof}

\section{Construction of the example}\label{sec:construction}
We will firstly construct our example on $\mathbb{T}^2\times[0,1]$, and then give the construction on $\mathbb{T}^3$ using a reversing technique.

Suppose $\mathbb{A}:\mathbb{T}^2\rightarrow\mathbb{T}^2$ has four fixed points, namely, $p,q,r$ and $s$. Let $\lambda>3$ be one of the two eigenvalues of $\mathbb{A}$. Denote by
\begin{displaymath}
P=(p,0.5),P_i=(p,i), Q_i=(q,i), R_i=(r,i), S_i=(s,i), i=0,1.
\end{displaymath}
From now on in this paper, we will take $\mathbb{T}^2$ to be the base space, and $\mathbb{T}^2\times[0,1]$ and $\mathbb{T}^3=\mathbb{T}^2\times\mathbb{S}^1$ will be viewed as the product spaces over $\mathbb{T}^2$. The points on $\mathbb{T}^2$ will be denoted by lower case letters, such as $p,q,r,s,x$. And the points on the product spaces will be denoted by corresponding capital letters, such as $P,Q,R,S,X$.

First of all, we need a local chart around point $p$ on $\mathbb{T}^2$ given by the next lemma.
\begin{lemma}\label{lem:adapted domain in 2d}
Consider the system $\mathbb{A}:\mathbb{T}^2\rightarrow\mathbb{T}^2$. There exist a homoclinic point of $p$, denoted by $a$, and a local chart centered at $p$, say $(U(p);(x_s,x_u))$, such that
\begin{itemize}
\item[(1)] $\mathbb{A}(x_s,x_u)=(\lambda^{-1}x_s,\lambda x_u),$ for every $(x_s,x_u)\in[-3,3]_s\times[-3,3]_u$.
\item[(2)] The coordinate of $p$ is $(0_s,0_u)$, and the coordinate of $a$ is $(0_s,1_u)$. There exists $n_0>0$ such that the coordinate of $\mathbb{A}^{2n_0}(a)$ is $(1_s,0_u)$.
\item[(3)] Denote $[-2,2]_s\times[-2,2]_u$ by $C$. Then $C\cap\mathbb{A}^{-2n_0}(C)$ contains two connected components, namely $C_1=[-2,2]_s\times[-2\lambda^{-2n_0},2\lambda^{-2n_0}]_u$ and $C_2=[-2,2]_s\times[1-2\lambda^{-2n_0},1+2\lambda^{-2n_0}]_u$. Moreover, $\mathbb{A}^{2n_0}(x_s,x_u)=(1_s+x_s\lambda^{-2n_0},(x_u-1_u)\lambda^{2n_0})$, for every $(x_s,x_u)\in C_2$.
\item[(4)] $C,\mathbb{A}(C_2),\ldots,$ $\mathbb{A}^{2n_0-1}(C_2)$ are pairwise disjoint.
\end{itemize}
\end{lemma}

\begin{proof}
We firstly choose a sufficiently small neighbourhood $U(p)$ of $p$ and some local chart on $U(p)$ such that item(1) is satisfied. This is trivial since $\mathbb{A}$ is linear Anosov. Now choose some homoclinic point $b$ of $p$, such that $b\in W_{loc}^u(p)$ and $\mathbb{A}^{2m}(b)\in W_{loc}^s(p)$, $m>0$. By taking iterations of $b$ and taking a linear transformation of the local chart, $b$ and $m$ can be chosen such that the following conditions are satisfied simultaneously
\begin{itemize}
\item $b=(0_s,1_u),\mathbb{A}^{2m}(b)=(1_s,0_u)$;
\item $\mathbb{A}(b),\ldots,\mathbb{A}^{2m-1}(b)\notin[-2,2]_s\times[-2,2]_u$.
\end{itemize}
Then for every $1\leq i\leq 2m-1$, we choose some open neighbourhood $U(\mathbb{A}^i(b))$ of $\mathbb{A}^i(b)$ such that $U(\mathbb{A}(b)),\ldots,U(\mathbb{A}^{2m-1}(b))$ and $[-2,2]_s\times[-2,2]_u$ are pairwise disjoint. Let $\varepsilon>0$ be some small positive number such that for every $1\leq i\leq 2m-1$, $U(\mathbb{A}^i(b))$ contains an $su$-box centered at $\mathbb{A}^i(b)$ of size $\varepsilon$, i.e.,
$$\bigcup_{x\in W^u_{\varepsilon}(\mathbb{A}^i(b))}W^s_{\varepsilon}(x)\subset U(\mathbb{A}^i(b)).$$
Lastly, take $a=\mathbb{A}^{-2m_1}(b)\in W^u_{\frac{\varepsilon}{2}}(p)$, $\mathbb{A}^{2n_0}(a)=\mathbb{A}^{2m_1+2m+2m_2}(b)\in W^s_{\frac{\varepsilon}{2}}(p)$, and by another linear transformation of the local chart, we assume $a=(0_s,1_u),\mathbb{A}^{2n_0}(a)=(1_s,0_u)$. Define $C,C_1$ and $C_2$ as in item(3), and we claim that item(4) is satisfied. To see this, it is sufficient to notice that $\mathbb{A}^i(C_2)$ is contained in the $su$-box centered at $\mathbb{A}^i(a)$ of size $\varepsilon$ for every $1\leq i\leq 2n_0-1$.
\end{proof}
\begin{figure}[htbp]
\centering
\includegraphics[width=10cm]{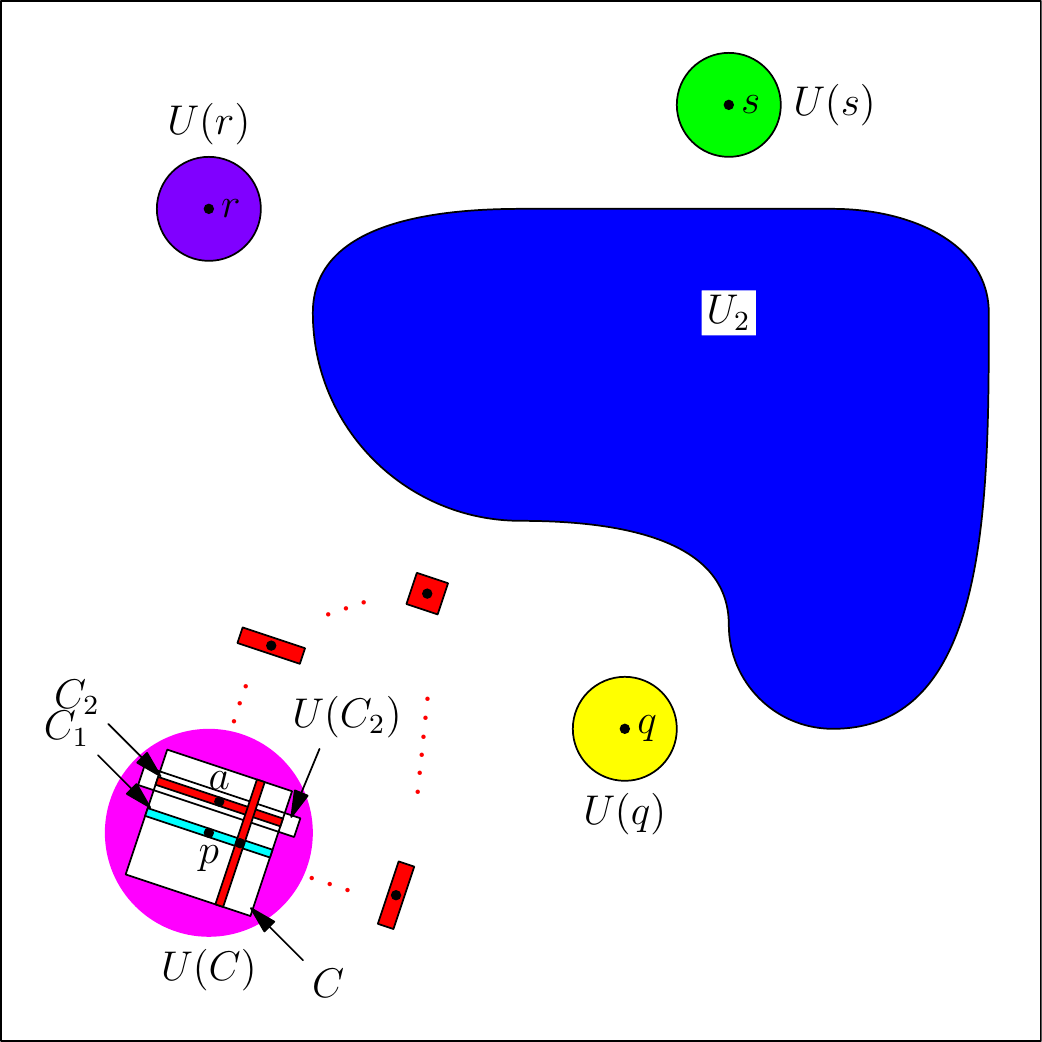}
\caption{Perturbation domains.}
\end{figure}
It can be easily deduced from the proof of the lemma that we can choose the local chart appropriately such that the Lebesgue measure of $\cup_{i=1}^{2n_0-1}\mathbb{A}^i(U(C_2))$ is smaller than any given positive number. Now we fix some $\varepsilon>0$ small enough, and choose a set of open sets contained in $\mathbb{T}^2$ such that
\begin{itemize}
\item[(1)] $U(x)\ni x$, and $\leb(U(x))<\varepsilon$, where $x=q,r,s.$
\item[(2)] $U(C)\supset C$ and $U(C_2)\supset C_2$, where $C,C_2$ are defined as in Lemma~\ref{lem:adapted domain in 2d}, and $U(C_2)\cap C_1=\varnothing$. Moreover, $$\leb(U(C)\bigcup\cup_{i=1}^{2n_0-1}\mathbb{A}^i(U(C_2)))<\varepsilon.$$
\item[(3)] $U_1\subset U_2,$ and $\leb(U_1)>0.5$.
\item[(4)] $U(r),U(s),U(q),U(C),U_2,\mathbb{A}(U(C_2)),\ldots,\mathbb{A}^{2n_0-1}(U(C_2))$ are pairwise disjoint.
\end{itemize}

\begin{remark}
The open sets above are the perturbation domains. In fact, we will perturb $\mathbb{A}\times\id:\mathbb{T}^2\times[0,1]\rightarrow\mathbb{T}^2\times[0,1]$ in $U(r)\times[0,1]$ and $U(s)\times[0,1]$ to get two fibers as in Definition~\ref{def:Kan's example with boundary} (K2). And perturbations will be made in $U_2\times[0,1]$ to guarantee (K4) in Definition~\ref{def:Kan's example with boundary}. As to  $U(C)\times[0,1]$ and $U(C_2)\times[0,1]$, they are used to construct a blender-horseshoe, which is the most important part in our example. Lastly, we will make perturbations in $U(q)\times[0,1]$ to get a hyperbolic fixed point which activates the blender constructed beforehand.
\end{remark}

Now we define a smooth vector field $\mathcal{X}$ on $\mathbb{T}^2\times[0,1]$, such that
\begin{equation}\label{equation:x}
\mathcal{X}(x,\theta)=\left\{\begin{array}{ll}
                             \alpha_1(x)\sin(\pi\theta)\frac{\partial}{\partial \theta},&x\in U(s),\\
                             -\alpha_1(x)\sin(\pi\theta)\frac{\partial}{\partial \theta},&x\in U(r),\\
                             -\alpha_1(x)\sin(2\pi\theta)\frac{\partial}{\partial\theta},&x\in U_2\cup U(C),\\
                             -\alpha_1(x)\beta_1(\theta)\sin(2\pi\theta)\frac{\partial}{\partial\theta},&x\in\cup_{i=0}^{2n_0-1}\mathbb{A}^i(U(C_2)),\\
                             0,&\textrm{otherwise},
                             \end{array}\right.
\end{equation}
where $\alpha_1\in C^{\infty}(\mathbb{T}^2)$ and $\beta_1\in C^{\infty}([0,1])$ are two bump functions satisfying
\begin{displaymath}
\alpha_1(x)\left\{\begin{array}{ll}
                 =1,&x\in\{r,s\}\cup U_1\cup C\bigcup\cup_{i=0}^{2n_0-1}\mathbb{A}^i(C_2),\\
                 =0,&x\notin U(r)\cup U(s)\cup U_2\cup U(C)\bigcup\cup_{i=0}^{2n_0-1}\mathbb{A}^i(U(C_2)),\\
                 \in[0,1],&\textrm{otherwise},
                 \end{array}\right.
\end{displaymath}
and
\begin{displaymath}
\beta_1(\theta)=\left\{\begin{array}{ll}
                       1,&\theta\in(0.5-\varepsilon,0.5+\varepsilon),\\
                       0,&\theta\in[0,\varepsilon)\cup(1-\varepsilon,1],\\
                       \in[0,1],&\textrm{otherwise.}
                       \end{array}\right.
\end{displaymath}

Denote by $\widehat{U(p)}=U(p)\times(0.5-\varepsilon,0.5+\varepsilon)$, and let $(\widehat{U(p)};(x_s,x_u,x_c))$ be a local chart defined on $\mathbb{T}^2\times[0,1]$  such that
\begin{itemize}
\item $(\widehat{U(p)};(x_s,x_u,x_c))$ is compatible with $(U(p);(x_s,x_u))$.
\item The coordinate of $P$ is $(0_s,0_u,0_c)$.
\item The vector field $\mathcal{X}$ can be wrote as $\mathcal{X}(x_s,x_u,x_c)=x_c\frac{\partial}{\partial x_c},$ for every $(x_s,x_u,x_c)\in[-3_s,3_s]\times[-3_u,3_u]\times[-3_c,3_c]$.
\end{itemize}

\begin{figure}[htbp]
\centering\includegraphics[width=10cm]{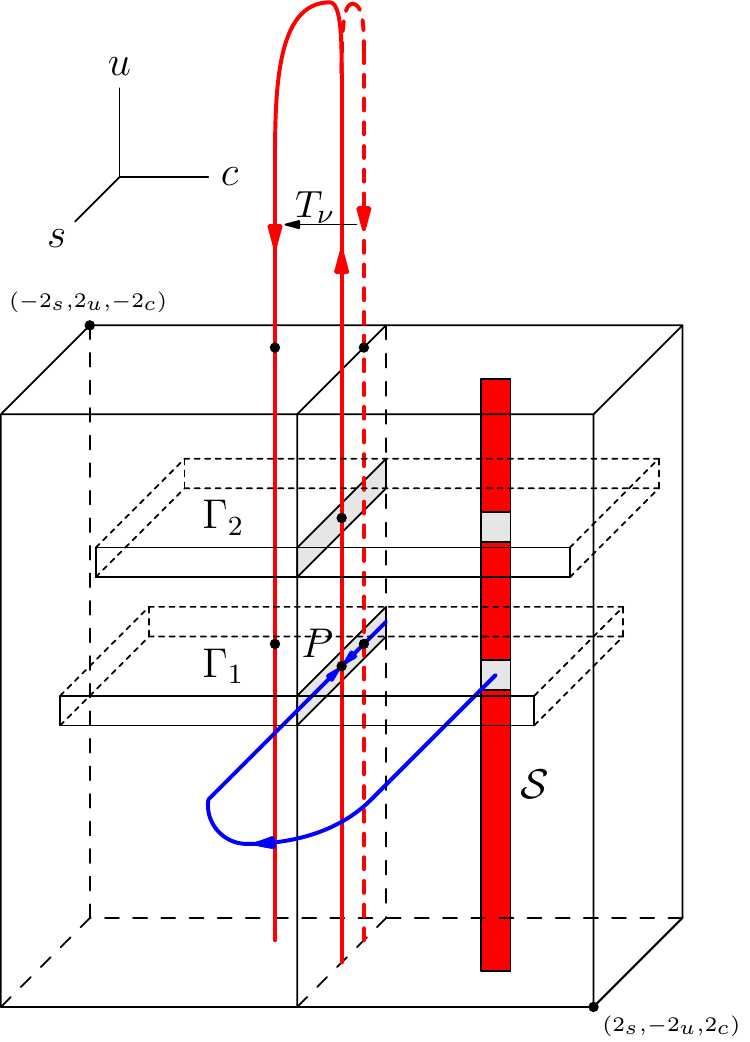}
\caption{Construction of blender-horseshoes.}
\end{figure}

Then define $\Gamma$ to be a subset of $\widehat{U(p)}$ such that it is identified with $[-2,2]_s\times[-2,2]_u\times[-2,2]_c$ under the local chart.

Now consider a family of $C^{\infty}$diffeomorphisms $\{f_{t,\nu}\}_{t,\nu\geq0}$ satisfying
\begin{itemize}
\item $f_{t,\nu}(x,\theta)=(\mathbb{A}\times\id)\circ\mathcal{X}_{t}(x,\theta)$, for all $x\in U(r)\cup U(s)\cup U(C)\cup U_2$, where $\mathcal{X}_t$ is the time $t$ map of $\mathcal{X}$.
\item $f_{t,\nu}^{2n_0}(x_s,x_u,x_c)=T_{\nu}\circ (\mathbb{A}\times\id)^{2n_0}\circ\mathcal{X}_{2n_0t}(x_s,x_u,x_c),$ for all $(x_s,x_u,x_c)\in C_2\times(0.5-0.5\varepsilon,0.5+0.5\varepsilon)$, where $T_{\nu}$ is the translation defined by $T_{\nu}(x_s,x_u,x_c)=(x_s,x_u,x_c-\nu).$
\item $f_{t,\nu}$ is skew product for all $t\geq0$ and $\nu\geq0$, that is, every $f_{t,\nu}$ preserves the center foliation $\mathscr{L}=\{\{x\}\times[0,1]:x\in\mathbb{T}^2\}$.
\item $f_{t,\nu}(x,\theta)=(\mathbb{A}x,\theta)$, for every $x\in\bigcup_{i=0}^{2n_0-1}\mathbb{A}^i(U(C_2))$ and $\theta\in[0,\varepsilon)\cup(1-\varepsilon,1]$.
\end{itemize}

\begin{lemma}\label{lem:parameters of blenderhorseshoe}
Let $\nu_t=e^{2n_0t}-1$, and $\Lambda_t=\cap_{i\in\mathbb{Z}}f^{2n_0i}_{t,\nu_t}(\Gamma)$. Then for $t>0$ small enough, $(\Lambda_t,f^{2n_0}_{t,\nu_t})$ is a blender-horseshoe with reference cube $\Gamma$ and reference saddles $P=(0_s,0_u,0_c)$ and $O=(1_s/(1-\lambda^{-2n_0}),1_u/(1-\lambda^{-2n_0}),1_c)$.
\end{lemma}
\begin{remark}
It is worthwhile pointing out that here the reference cube and reference saddles of the blender-horseshoes are independent of $t$.
\end{remark}

Define $\Gamma_1=[-2,2]_s\times[-2\lambda^{-2n_0},2\lambda^{-2n_0}]_u\times[-2e^{-2n_0t},2e^{-2n_0t}]_c$, $\Gamma_2=[-2,2]_s\times[1-2\lambda^{-2n_0},1+2\lambda^{-2n_0}]_u\times[1-3e^{-3n_0t},1+e^{2n_0t}]_c$. Then $f^{-2n_0}_{t,\nu_t}(\Gamma)\cap\Gamma=\Gamma_1\cup\Gamma_2,$ and
\begin{displaymath}
f^{2n_0}_{t,\nu_t}(x_s,x_u,x_c)=\left\{\begin{array}{lc}
                             (\lambda^{-2n_0}x_s,\lambda^{2n_0}x_u,e^{2n_0t}x_c),&(x_s,x_u,x_c)\in\Gamma_1.\\
                             (\lambda^{-2n_0}x_s+1,\lambda^{2n_0}(x_u-1),e^{2n_0t}(x_c-1)+1),&(x_s,x_u,x_c)\in\Gamma_2.\end{array}\right.
\end{displaymath}

Then it is easy to check that $P\in\Gamma_1$ and $O\in\Gamma_2$ are two hyperbolic fixed points of $f^{2n_0}_{t,\nu_t}$.

We will not go into details of the proof of Lemma~\ref{lem:parameters of blenderhorseshoe}, and refer the reader to Proposition~5.1 of \cite{BD3}. Instead, we will prove the following lemma just to exhibit the main properties of the blender-horseshoe.
\begin{lemma}
For every $t>0$ small enough, $W^s(P,f^{2n_0}_{t,\nu_t})$ intersects every strip $\mathcal{S}$ in between $W^s_{loc}(P,f^{2n_0}_{t,\nu_t})$ and $W^s_{loc}(O,f^{2n_0}_{t,\nu_t})$.
\end{lemma}

\begin{proof}
For any $t>0$ small enough, we choose the corresponding conefields $\mathcal{C}^s_{\varepsilon_t},\mathcal{C}^u_{\varepsilon_t},\mathcal{C}^{uu}_{\varepsilon_t}$ of size $\varepsilon_t$ sufficiently small, such that for any $v\in\mathcal{C}^u_{\varepsilon_t}$, $\parallel Df^{2n_0}_{t,\nu_t}(v)\parallel>(e^{2n_0t}+1)/2\parallel v\parallel.$ Denote by $\lambda'_t=(e^{2n_0t}+1)/2$.

To prove the lemma, we only have to show that for any vertical strip $\mathcal{S}$ in between $W^s_{loc}(P,f^{2n_0}_{t,\nu_t})$ and $W^s_{loc}(O,f^{2n_0}_{t,\nu_t})$, $f^{2n_0}_{t,\nu_t}(\mathcal{S}\cap(\Gamma_1\cup\Gamma_2))$ contains a vertical strip $\mathcal{S}'$, such that either $\width(\mathcal{S}')>\lambda'_t\width(\mathcal{S})$ or $\mathcal{S}'\cap W^s_{loc}(P,f^{2n_0}_{t,\nu_t})\neq \varnothing$.

To show the main ingredients of the proof, we firstly consider the vertical strips $\mathcal{S}$ in between $W^s_{loc}(P,f^{2n_0}_{t,\nu_t})$ and $W^s_{loc}(O,f^{2n_0}_{t,\nu_t})$ of the type $\{x_s\}\times[-2,2]_u\times[t_1,t_2]_c,$ where $0<t_1<t_2< 1$. If $t_2\leq e^{-2n_0t}$, denote $f^{2n_0}_{t,\nu_t}(\Gamma_1\cap\mathcal{S})$ by $\mathcal{S}'$, then $\mathcal{S}'=\{\lambda^{-2n_0}x_s\}\times[-2,2]_u\times[e^{2n_0t}t_1,e^{2n_0t}t_2]_c$ is a vertical strip in between $W^s_{loc}(P,f^{2n_0}_{t,\nu_t})$ and $W^s_{loc}(O,f^{2n_0}_{t,\nu_t})$, with width $e^{2n_0t}\width(\mathcal{S})>\lambda'_t\width(\mathcal{S})$.

If $e^{-2n_0t}<t_2< 1$, denote $f^{2n_0}_{t,\nu_t}(\mathcal{S}\cap\Gamma_2)$ by $\mathcal{S}'$, then there are two subcases. If $e^{2n_0t}(t_1-1)+1\geq0,$ then $\mathcal{S}'=\{\lambda^{-2n_0}x_s+1\}\times[-2,2]_u\times[e^{2n_0t}(t_1-1)+1,e^{2n_0t}(t_2-1)+1]_c$ is a vertical strip in between $W^s_{loc}(P,f^{2n_0}_{t,\nu_t})$ and $W^s_{loc}(O,f^{2n_0}_{t,\nu_t})$, with width $e^{2n_0t}\width(\mathcal{S})>\lambda'_t\width(\mathcal{S})$. Otherwise, noting that $e^{2n_0t}(t_2-1)+1>0,$ we deduce that $\mathcal{S}'$ is a vertical strip intersecting $W^s_{loc}(P,f^{2n_0}_{t,\nu_t})$.

Now consider a general vertical strip $\mathcal{S}$ in between $W^s_{loc}(P,f^{2n_0}_{t,\nu_t})$ and $W^s_{loc}(O,f^{2n_0}_{t,\nu_t})$, i.e., it is the image of an embedding $\Phi:[0,1]\times D^u\rightarrow \Gamma$ such that $\Phi(\{t\}\times D^u)$ is a vertical disk in between $W^s_{loc}(P,f^{2n_0}_{t,\nu_t})$ and $W^s_{loc}(O,f^{2n_0}_{t,\nu_t})$ for every $t\in[0,1]$, and $\Phi([0,1]\times D^u)$ is tangent to $\mathcal{C}^u_{\varepsilon_t}$.

If $\mathcal{S}\cap\Gamma_1\neq\varnothing$, and $\mathcal{S}\cap\{x\in\Gamma:x_c\geq e^{-2n_0t}\}=\varnothing$, then it is obvious to see that $\mathcal{S}'=f^{2n_0}_{t,\nu_t}(\mathcal{S}\cap\Gamma_1)$ is a vertical strip in between $W^s_{loc}(P,f^{2n_0}_{t,\nu_t})$ and $W^s_{loc}(O,f^{2n_0}_{t,\nu_t})$. Moreover, $\width(\mathcal{S}')>\lambda'_t\width(\mathcal{S})$ since $\mathcal{S}$ is tangent to $\mathcal{C}^u_{\varepsilon_t}$.

Otherwise, there exists $s_0\in(0,1)$, such that $\Delta_{s_0}=\Phi(\{s_0\}\times D^u)\subset\{x\in\Gamma:x_c\geq0.5\}$, and $\Delta'_{s_0}=f^{2n_0}_{t,\nu_t}(\Delta_{s_0}\cap\Gamma_2)$ is a vertical disk in between $W^s_{loc}(P,f^{2n_0}_{t,\nu_t})$ and $W^s_{loc}(O,f^{2n_0}_{t,\nu_t})$. Now define $I$ to be the maximal subinterval of $[0,1]$ containing $s_0$ such that for every $s\in I,$ $\Delta'_{s}=f^{2n_0}_{t,\nu_t}(\Delta_s\cap\Gamma_2)$ is a vertical disk in between $W^s_{loc}(P,f^{2n_0}_{t,\nu_t})$ and $W^s_{loc}(O,f^{2n_0}_{t,\nu_t})$. It is immediate that $I$ has nonempty interior by the continuity of $\Phi(\{\cdot\}\times D^u)$. If $I$ equals $[0,1]$, we are done. Otherwise, at lease one of the extremes of $I$, say $s_1$, is not contained in $I$. Then $f^{2n_0}_{t,\nu_t}(\Delta_{s_1})$ must have nonempty intersection with $W^s_{loc}(P,f^{2n_0}_{t,\nu_t})$ or $W^s_{loc}(O,f^{2n_0}_{t,\nu_t})$. The second subcase can not happen, since if $f^{2n_0}_{t,\nu_t}(\Delta_{s_1})\cap W^s_{loc}(O,f^{2n_0}_{t,\nu_t})\neq\varnothing,$ then $\Delta_{s_1}\cap W^s_{loc}(O,f^{2n_0}_{t,\nu_t})\neq\varnothing$ by computation, which contradict with the assumption that $\Delta_{s_1}$ is in between $W^s_{loc}(P,f^{2n_0}_{t,\nu_t})$ and $W^s_{loc}(O,f^{2n_0}_{t,\nu_t})$. Then the first subcase must hold, i.e., $f^{2n_0}_{t,\nu_t}(\Delta_{s_1})\cap W^s_{loc}(P,f^{2n_0}_{t,\nu_t})\neq\varnothing.$ The proof of the lemma has now been finished.
\end{proof}

From now on, we will denote $f_{t,\nu_t}$ by $f_t$ for simplicity. Define $\Gamma'=\{x\in\Gamma:0<x_c<1,-2<x_s<2\}$. Then it is easy to check that $\Gamma'$ is contained in the superposition domain $\mathbb{D}_t$ of the blender-horseshoe $(\Lambda_t,f^{2n_0}_{t,\nu_t})$ for every $t>0$ small enough. Since $\mathbb{A}$ is Anosov, then $W^u(\Gamma',\mathbb{A}\times\id)\cap \bar{q}$ contains an interval with nonempty interior, where $\bar{q}$ stands for the fiber $\{q\}\times[0,1]$. Choose such an interior point $Q=(q,\theta_0)$, and we remark here that $Q$ can be chosen very close to $(q,0.5)$. Then there exists a $C^1$-neighbourhood $\mathcal{U}$ of $\mathbb{A}\times\id$, such that for every $g\in\mathcal{U}$, $W^{uu}(Q,g)$ contains a vertical disk that belongs to the superposition region $\mathcal{D}_t$ of $(\Lambda_t,f^{2n_0}_{t,\nu_t})$.

Let $\gamma_t$ be a segment in $W^{uu}(Q,f_t)$ containing $Q$ and a vertical disk contained in the superposition region $\mathcal{D}_t$. Denote by $\gamma'_t$ the connected component of $\gamma_t\cap(U(q)\times[0,1])$ containing $Q$, and let $U'(q)\subset U(q)$ be a neighbourhood of $q$ such that $\gamma_t\cap (U'(q)\times[0,1])=\gamma'_t.$

Now let $\alpha_2:\mathbb{T}^2\rightarrow[0,1]$ be a smooth function such that $\alpha_2(q)=1$ and $\alpha_2(x)=0$ for all $x\notin U'(q).$ Let $\beta_2:[0,1]\rightarrow\mathbb{R}$ be a smooth function satisfying
\begin{itemize}
\item $\beta_2(0)=\beta_2(1)=\beta_2(\theta_0)=0.$
\item $\beta_2(\theta)>0$, when $0<\theta<\theta_0$, and $\beta_2(\theta)<0$, when $\theta_0<\theta<1$.
\item $\beta'_2(\theta)=1,$ when $\theta\in[0,\varepsilon)\cup(1-\varepsilon,1].$
\end{itemize}

\begin{figure}[htbp]\label{fig:global}
\centering
\includegraphics[width=10cm]{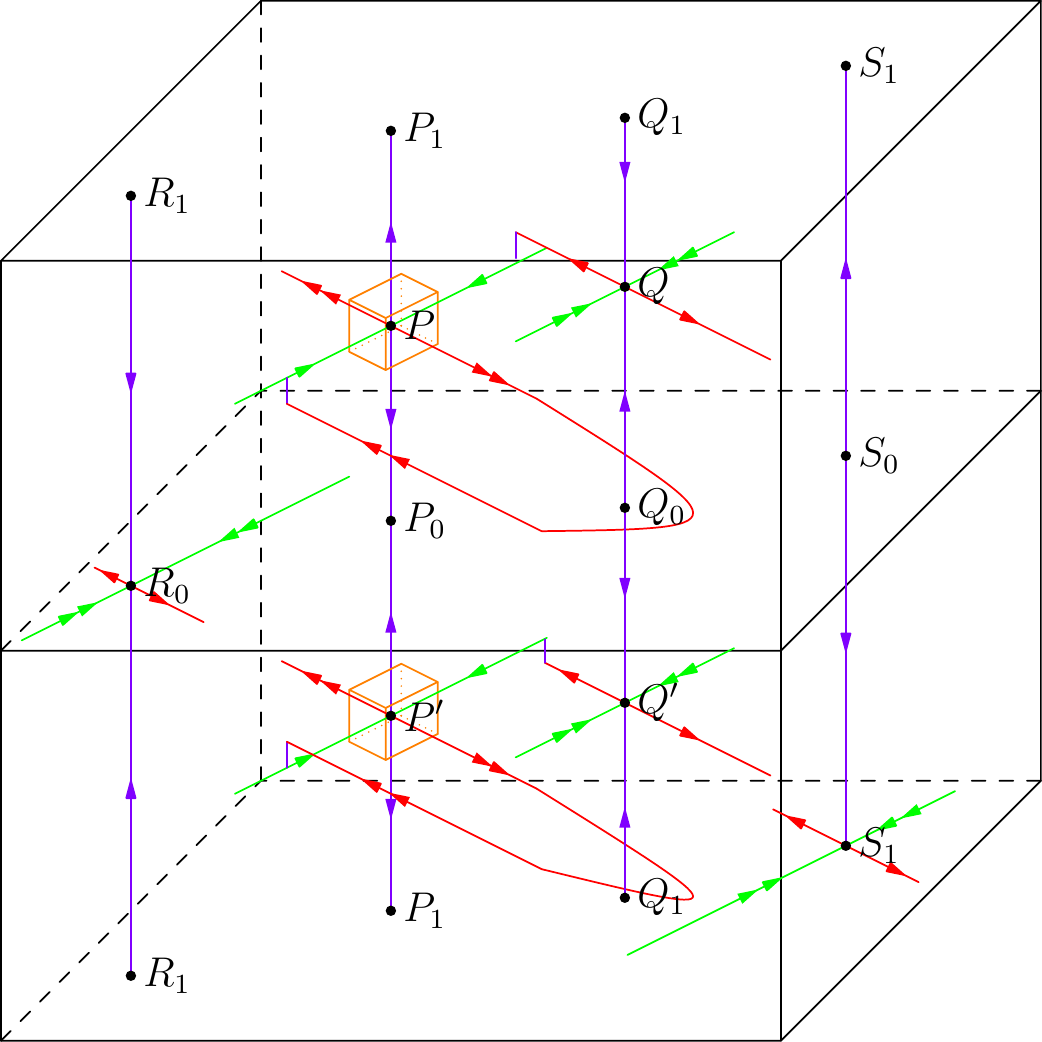}
\caption{Robustly transitive Kan's example on $\mathbb{T}^3$.}
\end{figure}

Define on $\mathbb{T}^2\times[0,1]$ a smooth vector field $\mathcal{Y}(x,\theta)=\alpha_2(x)\beta_2(\theta)\frac{\partial}{\partial\theta}$.

\begin{proposition}\label{prop:construction of kan on boundary manifolds}
For any $t>0$ small enough, we have $\tilde{f}_t=f_t\circ\mathcal{Y}_{t}\in\mathscr{K}^{\infty}(\mathbb{T}^2\times[0,1]).$
\end{proposition}
\begin{proof}
Denote by $\tilde{f}_t(x,\theta)=(\mathbb{A}x,\phi_{x,t}(\theta))$. Notice that the supports of $\mathcal{X}$ and $\mathcal{Y}$ are disjoint, the fiber diffeomorphisms $\phi_{x,t}$ can be written as the following for $\theta\in[0,\varepsilon)\cup(1-\varepsilon,1]$
\begin{equation}\label{equation:phi}
\phi_{x,t}(\theta)=\left\{\begin{array}{ll}
                         \mathcal{X}_{t}(x,\theta),& \textit{when}\ x\in U(r)\cup U(s)\cup U(C)\cup U_2\\
                         \mathcal{Y}_{t}(x,\theta),& \textit{when}\ x\in U(q).\\
                         \theta,&\textrm{otherwise}.
                         \end{array}\right.
\end{equation}
Now we are going to check the conditions $(K1)$--$(K4)$ in Definition~\ref{def:Kan's example with boundary} one by one.

By definition, for $i=0,1$, $\mathcal{X}(x,i)=0$ and $\mathcal{Y}(x,i)=0$, so $\phi_{x,t}(i)=i$. Hence condition $(K1)$ holds.

$\mathcal{X}(r,\theta)=-\sin(\pi\theta)\frac{\partial}{\partial \theta}$, so $\phi_{r,t}$ has exactly two fixed points, a sink at $\theta=0$ and a source at $\theta=1$. Similarly, $\mathcal{X}(s,\theta)=\sin(\pi\theta)\frac{\partial}{\partial \theta}$, so $\phi_{s,t}$ has exactly two fixed points, a sink at $\theta=1$ and a source at $\theta=0$. Hence condition $(K2)$ holds.

When $x\in U(r)\cup U(s)\cup U(C)\cup U_2$, $|\frac{\partial}{\partial\theta}\phi_{x,t}(\theta)|=\exp\int_0^t\frac{\partial}{\partial\theta}\mathcal{X}(x,\phi_{x,s}(\theta))ds\sim\exp(Ct)\rightarrow1(t\rightarrow0).$ Similar results hold when $x\in U(q)$. Thus $\tilde{f}_t$ is partially hyperbolic when $t$ is close to $0$. Hence condition $(K3)$ holds.

\begin{align}\nonumber
\int_{\mathbb{T}^2}\log\left|\frac{\partial}{\partial\theta}\phi_{x,t}(0)\right|dx&=\int_{U(r)}-\pi\alpha(x)tdx+\int_{U(s)}\pi \alpha(x) tdx+\int_{U_2\cup U(C)}-2\pi\alpha(x)tdx\\
&<\pi t(\leb(U(s)-2\leb(U_1)))<\pi t(\varepsilon-1)<0.\nonumber
\end{align}
Similarly, $\int_{\mathbb{T}^2}\log\left|\frac{\partial}{\partial\theta}\phi_{x,t}(1)\right|dx<\pi t(\varepsilon-1)<0$. Hence condition $(K4)$ holds. Now the proof of the proposition has been completed.
\end{proof}

Lastly, we will define a family of $C^{\infty}$-diffeomorphisms $\{K_t:\mathbb{T}^3\rightarrow\mathbb{T}^3\}_{t\geq0}$ by a reversing skill. Recall that $\mathbb{T}^3=\mathbb{T}^2\times\mathbb{S}^1$, and $\mathbb{S}^1=\mathbb{R}/2\mathbb{Z}$, we extend $\tilde{f}_t$ to $\hat{f}_t:\mathbb{T}^3\rightarrow\mathbb{T}^3$ by
\begin{equation}\label{equation:hat f}
\hat{f}_t(x,\theta)=\left\{\begin{array}{ll}
                 \tilde{f}_t(x,\theta)=(\mathbb{A}x,\phi_{x,t}(\theta)),&\theta\in[0,1]\\
                 (\mathbb{A}x,-\phi_{x,t}(-\theta)),&\theta\in[-1,0]
                 \end{array}\right.
\end{equation}
It is easy to see that $\hat{f}_t$ is the odd extension of $\tilde{f}_t$. And we remark here that $\hat{f}_t(t\geq0)$ are all smooth diffeomorphisms. This is because the odd extensions of $\mathcal{X}$ and $\mathcal{Y}$(denoted by $\hat{\mathcal{X}}$ and $\hat{\mathcal{Y}}$ respectively) are both smooth vector fields on $\mathbb{T}^3$, and $\hat{f}_t$ can be viewed as the smooth perturbations(supported on $\bigcup_{i=1}^{2n_0-1}\mathbb{A}^i(U(C))\times(0.5-\varepsilon,0.5+\varepsilon)$) of $\hat{\mathcal{X}}_t\circ\hat{\mathcal{Y}}_t$. At last, we define our examples $K_t(t\geq0)$ as the composition of $R=\id_{\mathbb{T}^2}\times-\id_{\mathbb{S}^1}:\mathbb{T}^2\times\mathbb{S}^1\rightarrow\mathbb{T}^2\times\mathbb{S}^1$ and $\hat{f}_t(t\geq0)$, i.e., $K_t=R\circ\hat{f}_t$. According to equation~\ref{equation:hat f}, we have
\begin{equation}\label{equation:K_t}
K_t(x,\theta)=\left\{\begin{array}{lr}
                     (\mathbb{A}x,-\phi_{x,t}(\theta)),&0\leq\theta\leq1.\\
                     (\mathbb{A}x,\phi_{x,t}(-\theta)),&-1\leq\theta\leq0.
                     \end{array}\right.
\end{equation}

From the arguments above and Proposition~\ref{prop:construction of kan on boundary manifolds}, we get the following proposition as wanted.
\begin{proposition}\label{prop:construction of kan on boundaryless manifold}
For any $t>0$ small enough, $K_t\in\mathscr{K}^{\infty}(\mathbb{T}^3)$.
\end{proposition}

\section{Proof of the Theorem~\ref{theo:main theorem}}\label{sec:proofs}
Before we move on to the detailed proofs, we summarize here the main properties of the examples $\{K_t\}_{t\geq0}$ constructed in the previous section(see figure~\ref{fig:global}):
\begin{enumerate}
\item $K_0=\mathbb{A}\times-\id_{\mathbb{S}^1}.$ And for any $t>0$, $K_t\in\mathscr{K}^{\infty}(\mathbb{T}^3)$, i.e. $K_t$ satisfies item(K1)--(K4) in Definition~\ref{def:Kan's example on boundaryless manifolds}.
\item For any $t>0$, $K_t$ is a skew product reversing the orientation of the fibers, see equation~\ref{equation:K_t}.
\item For any $t>0$, $K_t$ admits a fixed fiber $\bar{p}=\{p\}\times\mathbb{S}^1$ which is Morse--Smale with exactly four fixed points, two sinks $P_0=(p,0),P_1=(p,1)$ and two sources $P=(p,0.5),P'=(p,-0.5)$. Moreover $K_t$ admits a blender-horseshoe with reference cube $\Gamma$(a small cube around $P$, contained in $\mathbb{T}^2\times(0,1)$) and reference saddles $P,O$, where $\Gamma,P,O$ are all independent of $t>0$. Symmetrically, $K_t$ admits a blender-horseshoe with reference cube $\Gamma'$(a small cube around $P'$, contained in $\mathbb{T}^2\times(-1,0)$) and reference saddles $P',O'$, where $\Gamma',P',O'$ are all independent of $t>0$.
\item For any $t>0$, $K_t$ admits a fixed fiber $\bar{q}=\{q\}\times\mathbb{S}^1$ which is Morse--Smale with exactly four fixed points, two sources $Q_0=(q,0),Q_1=(q,1)$ and two sinks $Q=(q,\theta_0),Q'=(q,-\theta_0)$. Moreover, $W^u(Q,K_t)$ contains a $u$-disk which belongs to the superposition region of the $cu$-blender $\Lambda_t=\cap_{n\in\mathbb{Z}}K_t^n(\Gamma)$. Symmetrically, $W^u(Q',K_t)$ contains a $u$-disk which belongs to the superposition region of the $cu$-blender $\Lambda'_t=\cap_{n\in\mathbb{Z}}K_t^n(\Gamma')$.

\end{enumerate}
\subsection{Robust transitivity}
In this subsection, we will prove that the examples constructed in the previous section is robustly transitive. The following criterion will be used.
\begin{lemma}[\cite{BDV}, Lemma~7.3]\label{lem:criterium of robust transitivity}
Let $f:M\rightarrow M$ be a diffeomorphism, and $p$ be a hyperbolic periodic point. Suppose that invariant manifolds of $\orb(p)$ are both (robustly) dense in $M$. Then $f$ is (robustly) transitive.
\end{lemma}

The following lemma can be deduced from Theorem~7.1 of \cite{HPS} directly.
\begin{lemma}[\cite{HPS}]\label{lem:openness of skew product}
Let $M,N$ be compact boundaryless manifolds, and $F:M\times N\rightarrow M\times N, (x,y)\mapsto(f(x),\phi_x(y))$ be a normally hyperbolic skew--product $C^1$-diffeomorphism. Assume that $f:M\rightarrow M$ is Anosov, then there exists a $C^1$-neighbourhood $\mathcal{U}$ of $F$ in $\diff^1(M\times N)$, such that for any $G\in\mathcal{U}$,
\begin{enumerate}
\item $G$ is conjugate by $\Phi_G$ to a skew--product $G^*:(x,y)\rightarrow(f(x),\phi^G_x(y))$, which depends continuously on $G$.
\item For any $x\in M$, $\Phi_G^{-1}(\{x\}\times N)$ is diffeomorphic to $N$, and $C^1$-close to $\Phi_F^{-1}(\{x\}\times N)=\{\phi^{-1}(x)\}\times N$.
\end{enumerate}
\end{lemma}

Recall that for any $x\in\mathbb{T}^2$, we define $\bar{x}$ to be the fiber $\{x\}\times\mathbb{S}^1$. Using Lemma~\ref{lem:openness of skew product}, there exists a $C^1$-neighbourhood $\mathscr{U}_t$ of $K_t$ such that for any $g\in\mathcal{U}_t$ and any $x\in\mathbb{T}^2$, we can define the continuation of $\bar{x}$ to be $\Phi_g^{-1}(\{x\}\times\mathbb{S}^1)$, and denote it by $\bar{x}_g$.

\begin{proposition}\label{prop:robust transitivity}
Let $\{K_t\}_{t\geq0}$ be as above. Then for every $t>0$ small enough, $K_t$ is robustly transitive, i.e., there exists a $C^1$-neighbourhood $\mathcal{U}_t$ of $K_t$, such that every $g\in\mathcal{U}_t$ is transitive.
\end{proposition}
\begin{proof}
We will prove the proposition by using the argument of Lemma~\ref{lem:criterium of robust transitivity}. In fact, we only have to show that
\begin{claim}
$W^u(\orb(P),K_t)$ and $W^s(\orb(P),K_t)$ are both robustly dense in $\mathbb{T}^3$.
\end{claim}
\noindent\emph{Proof of the claim:} Denote $P'=(p,-0.5)$, then $\orb(P)=\{P,P'\}$. Similarly, denote $Q'=(q,-\theta_0)$, then $\orb(Q,K_t)=\{Q,Q'\}$. Firstly, note that
$$\overline{W^u(\orb(P),K_t)}=\overline{W^u(P,K_t)}\cup\overline{W^u(P',K_t)}=\overline{W^u(\bar{p},K_t)}=\mathbb{T}^3.$$
Since $K_t$ is normally hyperbolic skew product diffeomorphism, we can use Lemma~\ref{lem:openness of skew product} to conclude that there exists a $C^1$-neighbourhood $\mathcal{U}_t$ of $K_t$ such that $\overline{W^u(\bar{p}_g,g)}=\mathbb{T}^3$ still holds for every $g\in\mathcal{U}_t,$ and hence $\overline{W^u(\orb(P_g),g)}=\mathbb{T}^3$. Analogous argument shows that $\overline{W^s(\bar{q}_g,g)}=\mathbb{T}^3$ holds for every $g\in\mathcal{U}_t,$ and hence $\overline{W^s(\orb(Q_{g}),g)}=\mathbb{T}^3.$

According to the construction, $(\Lambda_t,K_t^{2n_0})$ is a blender-horseshoe with reference cube $\Gamma$ and reference saddles $P$ and $O$. And $W^u(Q,K_t^2)$ contains a vertical disk that belongs to the superposition region $\mathcal{D}_t$. Then by using of Lemma~\ref{lem:blenderhorseshoe and cu blender} and Lemma~\ref{lem:blender}, we have $W^s(Q_{g},g^2)\subset \overline{W^s(P_g,g^2)}.$ Symmetrically, we have $W^s(Q'_{g},g^2)\subset\overline{W^s(P'_g,g^2)}.$ Therefore,
\begin{align}
\overline{W^s(\orb(P_g),g)}&=\overline{W^s(P_g,g^2)}\cup\overline{W^s(P'_g,g^2)}\nonumber\\
                           & \supset\overline{W^s(Q_{g},g^2)}\cup\overline{W^s(Q'_{g},g^2)}\nonumber\\
                           &=\overline{W^s(\orb(Q_{g}),g)}=\mathbb{T}^3.\nonumber
\end{align}

The proof of the claim has now been finished, thus the proof of the proposition has been finished too.
\end{proof}

\subsection{Having two physical measures means Kan--like}
In this subsection, we will prove Theorem \ref{theo:two-pm}. First we need a lemma for Kan's examples on $\mathbb{T}^3$.

\begin{lemma}\label{lem:trans-inters}
For every $f\in\mathscr{K}^{1+}(\mathbb{T}^3)$, denote $R_0=(r,0)$ and $S_1=(s,1)$ to be two fixed point of $f$ with indices 2 from the definition (K2) of Kan's examples. Then there exists a $C^1$-neighborhood ${\cal U}$ of $f$, and a constant $L>2$, such that for any $g\in{\cal U}$ and $X\in\mathbb{T}^3$, we have
$$
W^{uu}_L(X,g)\pitchfork W^s_L(R_{0,g},g)\neq\varnothing, \qquad or \qquad
W^{uu}_L(X,g)\pitchfork W^s_L(S_{1,g},g)\neq\varnothing.
$$
\end{lemma}

\begin{proof}
From the definition of $f$, we only need to take the constant $L$ large enough, then at least one of these two properties holds. Since the transverse intersection is an open property, and the stable(unstable) manifolds of uniform size $L$ vary continuously with respect to the diffeomorphism, we have proved this lemma.
\end{proof}

\begin{lemma}\label{lem:kan is mcc}
For every $f\in\mathscr{K}^{1+}(\mathbb{T}^3)$, $f$ has mostly contracting center.
\end{lemma}

\begin{proof}
This is a well-known result, and we write down its proof for completeness. Let $\mu_0$ and $\mu_1$ be the normalized  Lebesgue measures restricted to the invariant tori $T_0=\mathbb{T}^2\times\{0\}$ and $T_1=\mathbb{T}^2\times\{1\}$, respectively. Since $\mathbb{A}:\mathbb{T}^2\rightarrow\mathbb{T}^2$ is a hyperbolic automorphism, $\mu_0$ and $\mu_1$ are the unique ergodic Gibbs $u$-states supported on $T_0$ and $T_1$, respectively. Recall that the center Lyapunov exponents of $\mu_0$ and $\mu_1$ are both negative, according to $(K4)$, so they are both physical measures by Proposition~\ref{prop:gibbs ustate basic properties}. According to Proposition~\ref{prop:mcc equivalent property}, to show $f$ has mostly contracting center, it suffices to show that $f$ has no other ergodic Gibbs $u$-states.

By contradiction, assume there exists an ergodic Gibbs $u$-state $\nu\notin\{\mu_0,\mu_1\}$. Then there exists a $u$-disk $D$ which intersect $B(\nu)$ on a full Lebesgue measure subset $D_0\subset D$. Since $\nu$ is an invariant measure, this also works for $f^n(D)$. According to lemma \ref{lem:trans-inters}, there exists $n$ large enough, such that $f^n(D)\pitchfork W^s_L(R_0)\neq\varnothing$  or $f^n(D)\pitchfork W^s_L(S_1)\neq\varnothing$. Without loss of generality, we assume the first case happens. Again, since $\mu_0$ is an ergodic Gibbs $u$-state, there exists a $u$-disk $D'\subset T_0$ and its subset $D'_0$ with positive Lebesgue measure such that every $X\in D'_0$ belongs to $B(\mu_0)$ and has 2-dimensional Pesin local stable manifold $W^s_{loc}(X)$ with uniform size. Moreover, the Pesin local stable manifolds vary absolutely continuously according to the classical Pesin theory.

Since $W^u(R_0)$ is dense in $T_0$, and $f^n(D)$ is arbitrarily close to $W^u(R_0)$ as $n$ tends to infinity, there exists $n_0$ large enough such that $f^{n_0}(D)$ cuts $\cup_{X\in D'_0}W^s_{loc}(X).$ The intersection is contained in $B(\mu_0)$ (since $\cup_{X\in D'_0}W^s_{loc}(X)\subset B(\mu_0)$), and has positive Lebesgue measure by the absolute continuity of $\{W^s_{loc}(X)\}_{X\in D'_0}$. This means $D$ itself has a positive Lebesgue measure subset contained in $B(\mu_0)$, which contradicts with the fact that $D$ has a full Lebesgue measure subset $D_0$ such that $D_0\subset B(\nu)$.
\end{proof}

\begin{proposition}\label{prop:two physical measures}
Let $K_t$ and $\mathcal{U}_t$ be defined as in Proposition~\ref{prop:robust transitivity}. Then by shrinking $\mathcal{U}_t$ if necessary, for any $g\in\mathcal{U}_t\cap\diff^{1+}(\mathbb{T}^3)$, if $g$ has two physical measures, then $g\in\mathscr{K}^{1+}_l(\mathbb{T}^3)$.
\end{proposition}

It is easy to see that Proposition~\ref{prop:two physical measures} can be deduced from Theorem~\ref{theo:two-pm} immediately. So we only have to prove Theorem~\ref{theo:two-pm}.

\begin{proof}[Proof of Theorem~\ref{theo:two-pm}]
We take the $C^1$-neighborhood ${\cal U}$ small enough, such that it satisfies Lemma \ref{lem:trans-inters} and every $g\in{\cal U}\cap\diff^{1+}(\mathbb{T}^3)$ has mostly contracting center. Moreover, we require that every $g\in{\cal U}\cap\diff^{1+}(\mathbb{T}^3)$ is still partially hyperbolic with uniformly compact center foliation.

Now assume $\mu_1$ and $\mu_2$ are two different ergodic Gibbs $u$-states of some $g\in{\cal U}\cap\diff^{1+}(\mathbb{T}^3)$. Then $\supp (\mu_1)$ and $\supp (\mu_2)$ are two disjoint compact $u$-saturated sets, according to Proposition~\ref{prop:mcc equivalent property}. We have the following claim.

\begin{claim}\label{claim:unique intersection}
For every $x\in\mathbb{T}^2$, $\bar{x}_g$ intersects $\supp (\mu_i)$ with exactly one point, for $i=1,2$.
\end{claim}

\begin{proof}[Proof of the claim]
Argue by contradiction. Assume that there exist $x\in\mathbb{T}^2$ and two points $X_1,X_2\in\bar{x}_g$, such that $X_1,X_2\in\supp (\mu_1)$. Notice that for any point $Y\in \supp (\mu_2)$, from Lemma \ref{lem:trans-inters}, we know that either $W_L^{uu}(Y,g)\pitchfork W_L^s(R_{0,g},g)\neq\varnothing$, or $W_L^{uu}(Y,g)\pitchfork W_L^s(S_{1,g},g)\neq\varnothing$. Without loss of generality, we assume the first case holds. Then by using of $\lambda$-lemma and the invariance of $\supp (\mu_2)$, we have $W^{uu}(R_{0,g},g)\subset \supp (\mu_2)$.

Since $X_1$ and $X_2$ are in the same center leaf, there exist $z\in\mathbb{T}^2$ and two points $Z_1,Z_2\in\bar{z}_g$, such that
\begin{displaymath}
Z_i\in W^{uu}(X_i,g)\cap W^{ss}(\bar{r}_g), \qquad i=1,2.
\end{displaymath}
However, noticing that $W^{ss}(R_{1,g},g)\cap\bar{z}_g$ consists of a unique point, we conclude that at least one of $Z_1$ and $Z_2$ belongs to $W^s(R_{0,g},g)$. Applying $\lambda$-lemma again, we have
$W^{uu}(R_{0,g},g)\subset \supp (\mu_1)$. This is a contradiction since $\supp (\mu_1)\cap \supp (\mu_2)=\varnothing$.
\end{proof}

From the proof of the claim, we can see that $\supp (\mu_i)$ can not intersects $W^s(R_{0,g},g)$ and $W^s(S_{1,g},g)$ simultaneously, where $i=1,2$. Therefore, following the same arguments as above, we can assume that $W^{uu}(R_{0,g},g)\subset \supp (\mu_1)$, and $W^{uu}(S_{1,g},g)\subset \supp (\mu_2)$.

Since $\supp (\mu_i)$ is compact, the map $x\in\mathbb{T}^2\mapsto\supp(\mu_i)\cap\bar{x}_g$ is lower semi-continuous, where $i=1,2$. However, Claim~\ref{claim:unique intersection} shows us each of the intersections in the maps consists of a unique point. This implies that the maps are both continuous. That is, $\supp (\mu_1)$ and $\supp (\mu_2)$ are two topological tori transverse to the center foliation. Combining with the assumption in the previous paragraph, we conclude
\begin{displaymath}
\overline{W^{uu}(R_{0,g},g)}=\supp (\mu_1), \qquad {\rm and} \qquad \overline{W^{uu}(S_{1,g},g)}=\supp (\mu_2).
\end{displaymath}

\begin{figure}[htbp]
\centering
\begin{overpic}{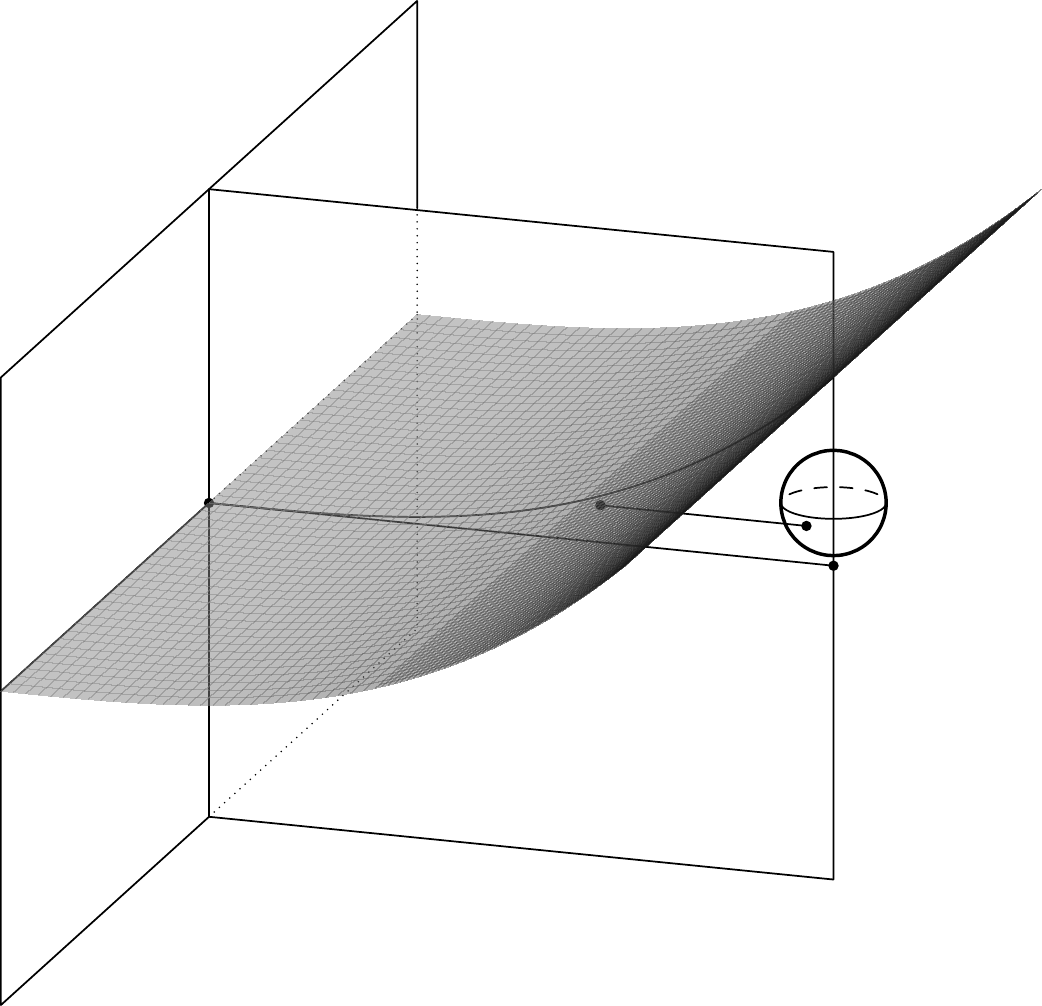}
\put(80,40){$Y_0$}
\put(85,50){$V$}
\put(20,50){$X_0$}
\put(76,47){$Y$}
\put(55,50){$X$}
\put(45,70){$W^c(W_{loc}^{ss}(X_0,g),g)$}
\put(30,85){$W^c(W_{loc}^{uu}(X_0,g),g)$}
\put(90,59){$\supp(\mu_1)$}
\put(90,60){\vector(-1,1){5}}
\put(51,53){$\sigma$}
\put(50,53){\vector(-1,-1){5}}
\end{overpic}
\caption{$\supp(\mu_1)$ is $s$-saturated.}
\end{figure}
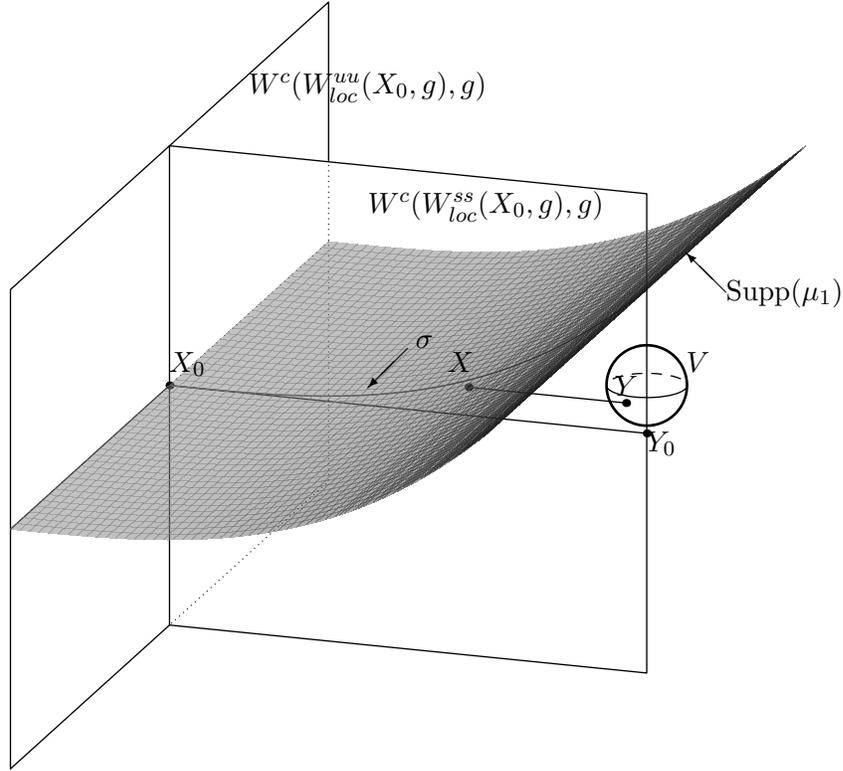

\begin{claim}
Both $\supp (\mu_1)$ and $\supp (\mu_2)$ are $s$-saturated, thus they are $su$-tori.
\end{claim}

\begin{proof}[Proof of the claim]
We only prove this claim for $\supp (\mu_1)$. Suppose $\supp (\mu_1)$ is not $s$-saturated, then there exist a point $X_0\in \supp (\mu_1)$ and a point $Y_0\in W^{ss}_{loc}(X_0,g)\cap \bar{y}_g$, such that $Y_0\notin \supp (\mu_1)$.  Since $\supp(\mu_1)$ is compact, there exists an open neighbourhood $U$ of $y$ in $\bar{y}_g$, such that $U\cap \supp (\mu_1)=\varnothing.$

Let $\sigma=W^c(W^{ss}_{loc}(X_0,g),g)\cap \supp (\mu_1)$, then it is a continuous segment. Now using the strong stable holonomy map, there exists a segment $l\subset \bar{y}_g$ such that $l=h^s_g(\sigma)$. Let $Z_0\in l\cap \mathscr{U}$ be an interior point of $l$. Now we can choose a sufficiently small neighbourhood $V$ of $Z_0$ in $\mathbb{T}^3$, such that for every point $Y\in V,$ there exists $X\in \supp (\mu_1)\cap W^{ss}(Y,g)$. This is because $\supp (\mu_1)$ is a $u$-saturated continuous surface.

Recall that both $W^s(R_{0,g},g)$ and $W^s(S_{1,g},g)$ are dense in $\mathbb{T}^3$, there exists two points $Y^0\in \mathscr{V}\cap W^s(R_{0,g},g)$, $Y^1\in \mathscr{V}\cap W^s(S_{1,g},g)$. So from $\lambda$-lemma, we know that both $R_{0,g}$ and $S_{1,g}$ belongs to $\supp (\mu_1)$. Since $\supp (\mu_2)$ is also $u$-saturated, so Lemma \ref{lem:trans-inters} and $\lambda$-lemma also implies at least one of $R_{0,g}$ and $S_{1,g}$ belongs to $\supp (\mu_2)$, which implies $\supp (\mu_1)\cap\supp (\mu_2)\neq\varnothing$.

This contradiction means $\supp (\mu_1)$ must be $s$-saturated.
\end{proof}

Now we move on to prove that the basins of $\mu_1$ and $\mu_2$ are intermingled. The ingredients of the proof are quite similar to Kan's examples, and we just sketch the main arguments.

Since $g$ also has mostly contracting center, the center Lyapunov exponent of $\mu_i$ is still negative. For any curve $\gamma$ which is transverse to the $E^{cs}_g$ direction and does not intersect with $\supp (\mu_1)$ nor $\supp (\mu_2)$, from Lemma \ref{lem:trans-inters}, up to take some forward iteration, that $\gamma$ crosses the local surfaces $W^s_{loc}(R_{0,g},g)$ and $W^s_{loc}(S_{1,g},g)$. Therefore,  for sufficient large $n$, $f^n(\gamma)$ will contain some disks sufficiently close to $\supp (\mu_1)$ and $\supp (\mu_2)$, and hence intersect with $B(\mu_1)$ and $B(\mu_2)$ both in a set of positive Lebesgue measure. Notice that $B(\mu_1)$ and $B(\mu_2)$ are both $g$-invariant, we conclude that $\gamma$ itself intersects with $B(\mu_1)$ and $B(\mu_2)$ both in a set of positive Lebesgue measure. Now the conclusion follows by using Fubini's Theorem.
\end{proof}

\subsection{Having a unique physical measure means mixing property}
In this subsection, we will prove that for the examples constructed in the previous section, if they admit a unique physical measure under perturbations, then the perturbed ones must be topologically mixing. Combined with Proposition \ref{prop:two physical measures}, the proof of Theorem~\ref{theo:main theorem} will be finished.

Recall that in the previous subsection, the perturbed systems remain to have two invariant tori. While in this subsection, we will show that if the perturbed systems admit a unique physical measure, then at least one of its invariant tori is broken. The following proposition provides us a characterization of the broken torus, which will be used later.

\begin{proposition}\label{prop:break invariant torus}
Let $K_t$ and $\mathcal{U}_t$ be defined as in Proposition~\ref{prop:robust transitivity}, $T$ be one of its two invariant tori, and $A\in\bar{a},B\in\bar{b}$ be any two hyperbolic periodic points contained in $T$. By shrinking $\mathcal{U}_t$ if necessary, assume that the continuation of $A$ and $B$ are well defined for any $g\in\mathcal{U}_t$. Then for any $g\in\mathcal{U}_t$,
\begin{itemize}
\item either $g$ admits an invariant torus $T_g$ which is $s,u$-saturated, and $T_g$ is close to $T$ in  Hausdorff metric,
\item or there exists $x\in W^u(a,\mathbb{A})\cap W^s(b,\mathbb{A})$ such that $W^{uu}(A_g,g)\cap\bar{x}_g\neq W^{ss}(B_g,g)\cap\bar{x}_g$.
\end{itemize}
\end{proposition}
\begin{remark}
If the first case happens, we call $T_g$ the \emph{continuation} of $T$. And if the second case happens, we say that the invariant torus $T$ is \emph{broken} under $g$.
\end{remark}
\begin{proof}
Assume that for every $x\in W^u(a,\mathbb{A})\cap W^s(b,\mathbb{A})$, we have $W^{uu}(A_g,g)\cap\bar{x}_g= W^{ss}(B_g,g)\cap\bar{x}_g$, we will prove that $g$ admits an invariant torus $T_g$ which is $s,u$-saturated and close to $T$ in  Hausdorff metric.

Denote $\Lambda_g=\overline{\{W^{uu}(A_g,g)\cap\bar{x}_g:x\in W^u(a,\mathbb{A})\cap W^s(b,\mathbb{A})\}}.$ For any $y\in\mathbb{T}^2$, we claim that $\#(\Lambda_g\cap \bar{y}_g)=1.$ Argue by contradiction, we suppose that there exists a point $y_0\in\mathbb{T}^2$, such that $\Lambda_g\cap\bar{y}_{0,g}$ contains at least two points, namely $Y^1$ and $Y^2$. By definition, there exist two sequence of points $\{x_n^i\in W^u(a,\mathbb{A})\cap W^s(b,\mathbb{A})\}_{n>0}(i=1,2)$ and $X_n^i\in\bar{x}_{n,g}^i\cap\Lambda_g$ such that $X_n^i$ tends to $Y^i$, $i=1,2$.
Suppose the center distance of $Y^1$ and $Y^2$ is bigger than some positive number, say $\delta$. We now choose $N\in\mathbb{N}$ large enough such that $\dd(X_N^i,Y^i)$ is much smaller that $\delta$, $i=1,2$. Then the distance of $x_N^1$ and $x_N^2$ in $\mathbb{T}^2$ is much smaller than $\delta$, and we can thus assume that there exists a point $z\in W^u_{loc}(x_N^1)\cap W^s_{loc}(x_N^2)$. Then using leaf conjugacy property, we see $W^{uu}_{loc}(\bar{x}_{N,g}^1,g)\cap W^{ss}_{loc}(\bar{x}_{N,g}^2,g)=\bar{z}_g.$ Due to the absolute continuity of the strong stable and strong unstable holonomy of $g$, there exist two points $Z^1,Z^2\in\bar{z}_g$ satisfying
\begin{itemize}
\item[(1)] $\{Z^1\}=W^{uu}_{loc}(X_{N,g}^1,g)\cap\bar{z}_g=W^{uu}(A_g,g)\cap\bar{z}_g,$
\item[(2)] $\{Z^2\}=W^{ss}_{loc}(X_{N,g}^2,g)\cap\bar{z}_g=W^{ss}(B_g,g)\cap\bar{z}_g,$
\item[(3)] $\dd_c(Z^1,Z^2)>\frac{\delta}{2}.$
\end{itemize}
This implies that $W^{uu}(A_g,g)\cap\bar{z}_g\neq W^{ss}(B_g,g)\cap\bar{z}_g$, which contradicts with our assumption. This contradiction implies our claim is true, i.e., $\#(\Lambda_g\cap \bar{y}_g)=1$ for every $y\in\mathbb{T}^2$. This means $\Lambda_g$ is actually a topological torus transverse to the center foliation, and we denote it by $T_g$ from now on. Bearing in mind $T_g$ contains a dense subset of $W^{uu}(A_g,g)$(resp. $W^{ss}(B_g,g)$), we get $T_g\supset W^{uu}(A_g,g)$(resp. $T_g\supset W^{ss}(B_g,g)$). Hence, $T_g$ is $u$-minimal(resp. $s$-minimal) and $u$-saturated(resp. $s$-saturated). $T_g$ is close to $T$ with respect to Hausdorff metric, because $\overline{W^{uu}(A_g,g)}$ varies lower semi-continuously as $g$ changes.
\end{proof}

Now we are ready to prove the main proposition of this subsection.
\begin{proposition}\label{prop:unique physical measure}
Let $K_t$ and $\mathcal{U}_t$ be defined as in Proposition~\ref{prop:robust transitivity}. Then by shrinking $\mathcal{U}_t$ if necessary, for any $g\in\mathcal{U}_t\cap\diff^{1+}(\mathbb{T}^3)$, if $g$ has a unique physical measure, then $g$ is robustly topologically mixing.
\end{proposition}
\begin{proof}
By shrinking $\mathcal{U}_t$ if necessary, suppose that the continuations of $R_0,S_1,P_0,Q_0,P_1,$ and $Q_1$ are all well defined.
We first claim that at least one of the two invariant tori of $K_t$ is broken under $g$. Otherwise, the two invariant tori $T_1\ni R_0$ and $T_2\ni S_1$ both have a continuation, denoted by $T_{1,g}$ and $T_{2,g}$ respectively. Since $T_{i,g}$ are both $u$-saturated and disjoint to each other, each of them support a Gibbs $u$-state, which is also a physical measure. This contradicts with the assumption that $g$ admits a unique physical measure.

Suppose now the invariant torus $T:=\mathbb{T}^2\times\{0\}$ is broken. To prove $g$ is topological mixing, we will prove both $W^s(P_g,g^2)$ and $W^u(P_g,g^2)$ are dense in $\mathbb{T}^3$.

Now that $T$ is broken, for $P_0$ and $Q_0$, there exists a point $x\in W^u(p,\mathbb{A})\cap W^s(q,\mathbb{A})$ such that $W^{uu}(P_{0,g},g)\cap\bar{x}_g\neq W^{ss}(Q_{0,g},g)\cap\bar{x}_g$. Then $W^s(Q_{0,g},g)=W^{ss}(Q_{0,g},g)$ intersects transversely with either $W^u(P_g,g^2)$ or $W^u(P'_g,g^2)$. We suppose $W^s(Q_{0,g},g)\cap W^u(P'_g,g^2)\neq\varnothing.$ Then $W^s(P'_g,g^2)\subset\overline{W^s(Q_{0,g},g)}$, due to Palis' $\lambda$-lemma. Recall that the existence of blender-horseshoes implies $W^s(Q_{g},g^2)\subset \overline{W^s(P_g,g^2)}$ and $W^s(Q'_{g},g^2)\subset \overline{W^s(P'_g,g^2)}$. Noting that $W^s(Q_{0,g},g)\subset \overline{W^s(Q_{g},g)}$, we conclude that
\begin{displaymath}
\overline{W^s(P_g,g^2)}\supset\overline{W^s(Q_g,g^2)}\supset\overline{W^s(Q_{0,g},g)}\supset\overline{W^s(P'_g,g^2)}\supset\overline{W^s(Q'_g,g^2)}.
\end{displaymath}
Therefore,
\begin{displaymath}
\overline{W^s(P_g,g^2)}\supset\overline{W^s(Q_{g},g^2)}\cup\overline{W^s(Q'_{g},g^2)}=\overline{W^s(\bar{q}_g,g)}=\mathbb{T}^3.
\end{displaymath}
Hence $\overline{W^s(P'_g,g^2)}=\overline{g(W^s(P_g,g^2))}=\mathbb{T}^3$.

Again, since $T$ is broken under $g$, for $P_0$ and $S_0$, there exists a heteroclinic point $y\in W^u(p,\mathbb{A})\cap W^s(s,\mathbb{A})$ such that $W^{uu}(P_{0,g},g)\cap\bar{y}_g\neq W^{ss}(S_{0,g},g)\cap\bar{y}_g$. Then $W^s(S_{0,g},g)$ intersects transversely with either $W^u(P_g,g^2)$ or $W^u(P'_g,g^2)$. We suppose $W^s(S_{0,g},g)\cap W^u(P_g,g^2)\neq\varnothing.$  And again, using Palis' $\lambda$-lemma, we get $W^u(S_{0,g},g)\subset\overline{W^u(P_g,g)}$, which immediately implies $\overline{W^u(P_g,g)}=\mathbb{T}^3.$

To show that $g$ is robustly mixing, it is sufficient to notice that if the invariant torus $T$ of $K_t$ is broken under $g$, then it is robustly broken.
\end{proof}

Now we can prove the corollary.

\begin{proof}[Proof of Corollary]
Fix some $t_0>0$ small enough. From the proof of Proposition \ref{prop:break invariant torus}, we see that we can take some sooth perturbation $g$ of $K_{t_0}$ such that one $su$-torus of $K_{t_0}$ has a continuation, and the other one is broken. Then $g$ is robustly mixing, according to Proposition~\ref{prop:unique physical measure}. However, the existence of an $su$-torus implies that neither the strong stable foliation nor the strong unstable foliation of $g$ is minimal.
\end{proof}

\begin{remark}
For the case where we require the diffeomorphism to preserve the orientation of the center foliation(see Remark~\ref{remark}), it is a little bit tricky. We just sketch the main ingredients of the proof. According to equation~\ref{equation:hat f}, for any $t>0$, $\hat{f}_t$ is a smooth function on $\mathbb{T}^3$ preserving the orientation of the center foliation. We now take $g$ to be a smooth perturbation of some $\hat{f}_t$ such that
\begin{itemize}
\item $g^{-1}\circ \hat{f}_t$ is a skew product diffeomorphism over $\id_{\mathbb{T}^2}:\mathbb{T}^2\rightarrow\mathbb{T}^2$;
\item $g$ preserves the orientation of the center foliation;
\item $g^{-1}\circ \hat{f}_t$ is supported on a small neighbourhood of the invariant torus $\mathbb{T}^2\times\{0\}$ of $\hat{f}_t$;
\item $W^s(Q_{0,g},g)\pitchfork W^u(P'_g,g)\neq\varnothing$ and $W^s(S_{0,g},g)\pitchfork W^u(P_g,g)\neq\varnothing$.
\end{itemize}
Now by the same arguments as in Proposition~\ref{prop:unique physical measure}, we can show that both $W^u(P_g,g)$ and $W^s(P_g,g)$ are robustly dense in $\mathbb{T}^3$. Hence $g$ is robustly topologically mixing. However, since $g$ still admits an invariant $su$-torus $\mathbb{T}^2\times\{1\}$, we know that neither the strong stable foliation nor the strong unstable foliation of $g$ is minimal.
\end{remark}

\noindent Cheng Cheng, School of Mathematical Sciences, Peking University, Beijing 100871, China\\
E-mail address: chcheng@pku.edu.cn
\vspace{0.5cm}

\noindent Shaobo Gan, School of Mathematical Sciences, Peking University, Beijing 100871, China\\
E-mail address: gansb@pku.edu.cn
\vspace{0.5cm}

\noindent Yi Shi, School of Mathematical Sciences, Peking University, Beijing 100871, China\\
E-mail address: shiyi@math.pku.edu.cn


\begin{thebibliography}{99}
\bibitem{AC} F. Abdenur, S. Crovisier, Transitivity and topological mixing for $C^1$ diffeomorphisms,
\textit{Essays in mathematics and its applications}, 1--16, Springer, Heidelberg, 2012.

\bibitem{ABV} J.F. Alves, C. Bonatti and M. Viana, SRB measures for partially hyperbolic diffeomorphisms whose central direction is mostly expanding, \textit{Invent. Math.}, 140: 351--398, 2000.
\bibitem{A} M. Anderson, Robust ergodic properties in partially hyperbolic dynamics, \textit{Trans. Amer. Math. Soc.}, vol. 362(4), 1831--1867, 2010.

\bibitem{BKR} P.G. Barrientos, Y. Ki and A. Raibekas, Symbolic blender--horseshoes and applications, \textit{Nonlinearity}, 27(12): 2805--2839, 2014.

\bibitem{BD1} C. Bonatti and L.J. D\'{i}az, Persistent nonhyperbolic transitive diffeomorphisms, \textit{Ann. of Math.(2)}, 143: 357--396, 1996.
\bibitem{BD2} C. Bonatti and L.J. D\'{i}az, Robust heterodimensional cycles and $C^1$-generic dynamics, \textit{J. Inst. Math. Jussieu}, 7: 469--525, 2008.
\bibitem{BD3} C. Bonatti and L.J. D\'{i}az, Abundance of $C^1$-robust homoclinic tangencies, \textit{Trans. Amer. Math. Soc.}, 364: 5111--5148, 2012.
\bibitem{BDU} C. Bonatti, L.J. D\'{i}az and R. Ures, Minimality of strong stable and unstable foliations for partially hyperbolic diffeomorphisms, \textit{J. Inst. Math. Jussieu} 1, no.4: 513--541, 2002.
\bibitem{BDV} C. Bonatti, L.J. D\'{i}az and M. Viana, Dynamics beyond uniform hyperbolicity, \textit{Encyclopaedia Math. Sci.}, vol. 102, Springer--Verlag, 2005.
\bibitem{BP} C. Bonatti and R. Potrie, Many intermingled basins in dimension 3, arXiv:1603.03803v1.

\bibitem{BV} C. Bonatti and M. Viana, SRB measures for partially hyperbolic diffeomorphisms whose central direction is mostly contracting, \textit{Isreal J. of Math.}, 115, 157--193, 2000.
\bibitem{BR} R. Bowen and D. Ruelle, The ergodic theory of Axiom A flows, \textit{Invent. Math.} 29, 181--202, 1975.

\bibitem{BHHTU} K. Burns, F.R. Herts, J.R. Herts, A. Talitskaya and R. Ures, Density of accessibility for partially hyperbolic diffeomorphisms with one-dimensional center, \textit{Discrete Contin. Dyn. Syst.}, 22(1--2), 75--88, 2008.

\bibitem{DGOY} S.W. McDonald, C. Grebogi, E. Ott and J.A. Yorke, Fractal basin boundaries, \textit{Phys. D.}, 17, 125--153, 1985.
\bibitem{DVY} D. Dolgopyat, M. Viana and J. Yang, Geometric and measure-theoretical structures of maps with mostly contracting center, \textit{Commun. Math. Phys.}, 341, 991--1014, 2016.

\bibitem{GS} S. Gan and Y. Shi, Topological mixing for Kan's map, In preparation. 

\bibitem{H} J.E. Hutchinson, Fractals and self-similarity, \textit{Indiana Univ. Math. J.}, 30(5): 713--747, 1981.
\bibitem{HN} A.J. Homburg and M. Nassiri, Robust minimality of iterated function systems with two generators, arXiv:1301.0400v1.
\bibitem{HPS} M. Hirsch, C. Pugh and M. Shub, \textit{Invariant manifolds}, Springer--Verlag, Birlin, Lecture Notes in Mathmatics, Vol. 583, 1977.
\bibitem{IKS} Y.S. Ilyashenko, V.A. Kleptsyn and P. Saltykov, Openness of the set of boundary preserving maps of an annulus with intermingled attracting basins, \textit{J. Fixed Point Theory Appl.}, 3(2): 449--463, 2008.
\bibitem{K} I. Kan, Open sets of diffeomorphisms having two attractors, each with an everywhere dense basin, \textit{Bull. Amer. Math. Soc.(N.S.)}, 31(1): 68--74, 1994.
\bibitem{KS} V.A. Kleptsyn and P.S. Saltykov, On $C^2$-stable effects of intermingled basins of attractors in classes of boundary-preserving maps, \textit{Trans. Moskov Math. Soc.}, 193--217, 2011.
\bibitem{NP} M. Nassiri and E. Pujals, Robust transitivity in Hamiltonian dynamics, \textit{Ann. Sci. \'Ec. Norm. Sup\'er.}, 45: 191--239, 2012.
\bibitem{O} A. Okunev, Milnor attractors of circle skew products, arXiv:1508.02132v1, 2015.
\bibitem{P} J. Palis, A global perspective for non-conservative dynamics, \textit{Ann. Inst. H. Poincar\'e Anal. Non Lin\'eaire}, 22(4): 485--507, 2005.
\bibitem{PS} Y. Pesin and Y. Sinai, Gibbs measures for partially hyperbolic attractors, \textit{Ergod. Th. \& Dynam. Sys.}, 324, 55--114, 1977.
\bibitem{R} D. Ruelle, A measure associated with Axiom A attractors, \textit{Amer. J. Math.}, 98, 619--654, 1976.
\bibitem{S} Ya. Sinai, Gibbs measures in ergodic theory, \textit{Russian Math. Surveys}, 27, 21--69, 1972.
\bibitem{UV} R. Ures and C.H. Vasquez, On the robustness of intermingled basins, arXiv:1503.07155v2, 2015.
\bibitem{VY} M. Viana and J. Yang, Physical measures and absolute continuity for one-dimensional center direction, \textit{Ann. Inst. H. Poincar\'e Anal. Non Lin\'eaire}, 30: 845--877, 2013.
\bibitem{Y} J. Yang, Entropy along expanding foliations, arXiv:1601.05504v1, 2016.
\end{thebibliography}
\end{document}